\documentclass[12pt]{amsart}       
\usepackage{txfonts}
\usepackage{amssymb}
\usepackage{eucal}
\usepackage{graphicx}
\usepackage{amssymb}
\usepackage{amsmath}
\usepackage{amscd}
\usepackage[all]{xy}           
\usepackage{amsfonts,latexsym}
\usepackage{xspace}
\usepackage{epsfig}
\usepackage{float}
\usepackage{mathrsfs} 
\usepackage{color}
\usepackage{fancybox}
\usepackage{colordvi}
\usepackage{multicol}
\usepackage{colordvi}
\usepackage{ifpdf}
\usepackage{iftex}
\ifXeTeX
  \usepackage{fontspec}
  \usepackage{xeCJK}
  \setCJKsansfont{Noto Sans CJK SC}
\else
\fi
\ifXeTeX
  \usepackage[colorlinks,final,backref=page,hyperindex]{hyperref}
\else
  \ifpdf
    \usepackage[colorlinks,final,backref=page,hyperindex]{hyperref}
  \else
    \usepackage[colorlinks,final,backref=page,hyperindex,hypertex]{hyperref}
  \fi
\fi
\usepackage[active]{srcltx} 

\makeatletter
\newcommand{\model@english@today}{%
  \ifcase\month\or January\or February\or March\or April\or May\or June%
  \or July\or August\or September\or October\or November\or December\fi
  \space\number\day, \number\year}
\let\today\model@english@today
\AtBeginDocument{%
  \let\today\model@english@today
}
\makeatother

\usepackage{tikz}

\usepackage{graphicx}

\usepackage{longtable}
\usepackage{booktabs}
\usepackage{array}

\usepackage{xcolor}
\usepackage[most]{tcolorbox}
\usepackage{enumitem}
\newtheorem{theorem}{Theorem}[section]

\newtheorem{proposition}[theorem]{Proposition}
\newtheorem{lemma}[theorem]{Lemma}

\newtheorem{corollary}[theorem]{Corollary}
\newtheorem{prop-def}{Proposition-Definition}[section]
\newtheorem{coro-def}{Corollary-Definition}[section]

\theoremstyle{definition}
\newtheorem{definition}[theorem]{Definition}
\newtheorem{remark}[theorem]{Remark}

\newtheorem{example}[theorem]{Example}

\newtheorem{assumption}[theorem]{Assumption}

\newcommand{\nc}{\newcommand}
\nc{\tred}[1]{\textcolor{red}{#1}}
\nc{\tblue}[1]{\textcolor{blue}{#1}}
\nc{\tgreen}[1]{\textcolor{green}{#1}}
\nc{\tpurple}[1]{\textcolor{purple}{#1}}
\nc{\btred}[1]{\textcolor{red}{\bf #1}}
\nc{\btblue}[1]{\textcolor{blue}{\bf #1}}
\nc{\btgreen}[1]{\textcolor{green}{\bf #1}}
\nc{\btpurple}[1]{\textcolor{purple}{\bf #1}}
\nc{\NN}{{\mathbb N}}
\nc{\ncsha}{{\mbox{\cyr X}^{\mathrm NC}}} \nc{\ncshao}{{\mbox{\cyr
X}^{\mathrm NC}_0}}

\newcommand{\delete}[1]{}

\nc{\mlabel}[1]{\label{#1}}
\nc{\mcite}[1]{\cite{#1}}
\nc{\mref}[1]{\ref{#1}}
\nc{\meqref}[1]{\eqref{#1}}
\nc{\mbibitem}[1]{\bibitem{#1}}

\delete{
\nc{\mlabel}[1]{\label{#1}{\hfill \hspace{1cm}{\bf{{\ }\hfill(#1)}}}}
\nc{\mcite}[1]{\cite{#1}{{\bf{{\ }(#1)}}}}
\nc{\mref}[1]{\ref{#1}{{\bf{{\ }(#1)}}}}
\nc{\meqref}[1]{\eqref{#1}{{\bf{{\ }(#1)}}}}
\nc{\mbibitem}[1]{\bibitem[\bf #1]{#1}}
}
\font\cyr=wncyr10 
\font\scyr=wncyr6
\nc{\sha}{{\mbox{\scyr X}}}
\nc{\shap}{{\mbox{\cyrs X}}} 
\nc{\shpr}{\diamond}    
\nc{\shp}{\ast} \nc{\shplus}{\shpr^+}
\nc{\shprc}{\shpr_c}    
\nc{\dep}{\mrm{dep}} \nc{\lc}{\lfloor} \nc{\rc}{\rfloor}
\nc{\db}{\leq_{\rm db}} \nc{\bfk}{\bf k}

\font\cyr=wncyr10 \font\cyrs=wncyr7
\nc{\li}[1]{\textcolor{red}{#1}}
\nc{\lir}[1]{\textcolor{red}{Li:#1}}
\nc{\yi}[1]{\textcolor{blue}{Yi: #1}}
\nc{\xing}[1]{\textcolor{purple}{Xing:#1}}
\nc{\revise}[1]{\textcolor{red}{#1}}

\nc{\Lin}{\operatorname{Lin}}  \nc{\Ad}{\operatorname{Ad}}
\providecommand{\etree}{\mathbf{1}}

\newtcolorbox{qshuStepBox}[1]{
  enhanced,
  breakable,
  colback=blue!2,
  colframe=blue!45!black,
  title={#1},
  fonttitle=\bfseries,
  boxrule=0.45pt,
  arc=1mm,
  left=1.2em,
  right=1.2em,
  top=0.7em,
  bottom=0.7em,
  before skip=0.9em,
  after skip=0.9em
}

\numberwithin{equation}{section}

\providecommand{\bZ}{\mathbf{Z}} \providecommand{\bX}{\mathbf{X}} \providecommand{\RR}{\mathbb{R}}
 \providecommand{\zuvij}{(\mathcal Z_{u,v}^{(a)})_{i,j}}
\providecommand{\zeuvi}{(\zeta_{u,v}^{(a)})_{i}} \providecommand{\crk}{{\rm CRK}} 

\NewDocumentCommand{\sint}{e{_^}}{\mathop{\textstyle\int}\nolimits_{\mkern-6mu #1}^{#2}}

\newcommand{\bigcdot}{\mathord{\textstyle\cdot}}
\begin{document}

\title[Controlled Runge-Kutta methods for CRDE]{Controlled B-series, controlled Runge-Kutta methods and controlled rough differential equations}

\author{Xinyuan An} 
\address{School of Mathematics and Statistics, Lanzhou University
Lanzhou, 730000, China
}
\email{anxy2026@Izu.edu.cn}

\author{Xingya Fan}
\address{College of Mathematics and System Sciences, Xinjiang University, 
Urumqi, 830017, China
}
\email{fanxingya@xju.edu.cn}

\author{Xing Gao$^{*}$}\thanks{*Corresponding author}
\address{School of Mathematics and Statistics, Lanzhou University
Lanzhou, 730000, China; Gansu Provincial Research Center for Basic Disciplines of Mathematics
and Statistics, Lanzhou, 730070, China
}
\email{gaoxing@lzu.edu.cn}

\date{\today}
\begin{abstract}
We develop controlled B-series and controlled Runge-Kutta methods as
extensions of classical B-series and Runge-Kutta theory to
controlled-driven rough differential equations. The driving signal is
a path controlled by an underlying $\alpha$-H\"older step-$2$ rough
path, with $\alpha\in(1/3,1/2]$, as arises when the output of one rough
system drives another. Controlled B-series retain the classical
rooted-tree and elementary-differential structure, while incorporating
driver-dependent coefficients defined recursively through controlled
rough integration. We establish finite-order expansions of the exact
solution and its controlled Runge-Kutta approximation, and derive
corresponding tree-based order conditions. Under suitable smoothness,
solvability, stability, and boundedness assumptions, matching through
tree order $p$ yields local error $O(|t-s|^{(p+1)\alpha})$ and global
error $O(h^{(p+1)\alpha-1})$, provided $(p+1)\alpha>1$. The standard
rough differential equation formulation is recovered when the
controlled driver is the reference path itself, and the classical
time-driven case recovers the usual B-series and Runge-Kutta methods.
We further construct a simplified method using only increments of the
controlled driver. Segmentwise canonical lifting of its piecewise
linear interpolation gives explicit tree coefficients and reduces the
third-order conditions to the classical ones. Assuming a Wong-Zakai
solution-approximation rate $r_0>0$, the simplified method has global
error $O(h^{\min\{r_0,(p+1)\alpha-1\}})$. Numerical experiments with
fractional Brownian controlled drivers illustrate the resulting
convergence behaviour.
\end{abstract}

\makeatletter
\@namedef{subjclassname@2020}{\textup{2020} Mathematics Subject Classification}
\makeatother
\subjclass[2020]{
60L20,
60L70,
65C30,
65L06,
65L20.
}

\keywords{Runge-Kutta methods, rough differential equations, controlled rough paths, B-series}

\maketitle

\tableofcontents

\setcounter{section}{0}

\allowdisplaybreaks

\section{Introduction}
\label{sec:intro}
Controlled rough differential equations, controlled B-series, and
controlled Runge-Kutta methods form the three components of this
paper. The equations describe systems driven by paths controlled by a
reference rough path. Controlled B-series organise their local solution
expansions, and controlled Runge-Kutta methods use the same rooted-tree
structure to formulate numerical order conditions. Our aim is to extend
the classical relation between differential equations, B-series, and
Runge-Kutta methods to this controlled-driver setting, including an
increment-based implementation and its convergence analysis.

\subsection{From RDE to CRDE}
\label{subsec:from-rde-to-crde}

Rough path theory, initiated by Lyons~\cite{Lyons1998}, provides a
pathwise framework for differential equations driven by irregular
signals; see \cite{FH2020,FV2010,LQ2002}. In the step-$2$ setting,
a rough differential equation (RDE) is driven by an enhanced signal
$\mathbf X=(X,\mathbb X)$ of $\alpha$-H\"older regularity, where
$\alpha\in(1/3,1/2]$. The first level $X$ is the immediate input,
while the second level $\mathbb X$ supplies the additional information
needed to define rough integration. The resulting equation is usually
written as $dY_t=f(Y_t)\,d\mathbf X_t$.

A controlled-driver formulation becomes natural when the input to a
system is obtained from another rough system. An intermediate response,
or a smooth observation of that response, need not coincide with $X$,
but its increments retain a precise relationship with those of $X$.
Gubinelli's controlled-path framework~\cite{Gubinelli2004} expresses
this relationship through a pair $\mathbf Z=(Z,Z')$, where $Z$ and
$Z'$ are $\alpha$-H\"older continuous and
$Z_{s,t}=Z'_sX_{s,t}+R^Z_{s,t}$ with a $2\alpha$-H\"older
remainder. The precise definition is recalled in
Definition~\ref{def:controlled-path}. Here $Z$ is the effective
input, and $Z'$ is its Gubinelli derivative relative to $X$; the
notation $\mathbf Z$ denotes a controlled pair rather than an
independently prescribed enhancement over $Z$.

We study controlled-driven rough differential equations (CRDEs),
written schematically as
\[
dY_t=F(Y_t)\,d\mathbf Z_t,
\qquad
Y_0=y_0,
\]
where $Z$ takes values in $\mathbb R^m$, $Y$ takes values in
$\mathbb R^e$, and
$F:\mathbb R^e\to\mathcal L(\mathbb R^m,\mathbb R^e)$.
Their precise meaning is given in
Definition~\ref{def:controlled-driven-rde} through the controlled
integral formulation \eqref{eq:controlled-driven-fixed-point}.
The reference rough path $\mathbf X$ determines the integration
structure, whereas $\mathbf Z$ supplies the driving data. The
well-posedness and continuity results used here are taken from
\cite{LiGao2026}. The ordinary RDE formulation is recovered when
$m=d$, $Z=X$, and $Z'=\operatorname{Id}_{\mathbb R^d}$.

A typical example is a cascade in which an upstream state satisfies
$dU_t=V(U_t)\,d\mathbf X_t$ and a downstream system is driven by
$Z_t=g(U_t)$ for a sufficiently smooth observation map $g$.
The controlled-path composition rule gives
$Z'_t=Dg(U_t)V(U_t)$; see \cite{FH2020,Gubinelli2004} and
Lemma~\ref{lem:composition}. Thus, the output of the upstream system
is an admissible controlled input for the downstream equation. Smooth
transformations $Z_t=g(X_t)$ provide a simpler instance of the same
mechanism. These examples motivate treating the effective driver
as an object in its own right, particularly when only that output,
rather than the full upstream state, is available to the numerical
method.

The controlled remainder is part of the driving information and
cannot simply be discarded. For example, in one dimension,
$Z_t=X_t+t$ is controlled by $X$ with derivative $Z'_t=1$ and
remainder $R^Z_{s,t}=t-s$, since $2\alpha\le1$. For
$F\equiv1$, its CRDE solution is $Y_t=y_0+X_t-X_0+t$.
Replacing the driver by its leading controlled component alone would
lose the drift term. The controlled description therefore preserves
more than the leading-order dependence on the reference signal.

The purpose of the CRDE formulation is not to place these equations
outside ordinary rough path theory. When an induced enhancement or
an enlarged state-space representation is available, the same dynamics
may also be described by a standard RDE. The direct controlled-driver
viewpoint is useful because it retains the structure of the supplied
input and allows its approximation to be analysed separately from
the numerical solution of the downstream equation. This motivates
both the controlled B-series construction and the controlled
Runge-Kutta methods developed below.

\subsection{Classical foundations and numerical challenges}
\label{subsec:intro-numerics}

The relation between rooted trees, B-series, and Runge-Kutta order
conditions is classical
\cite{Butcher1963,Butcher1972,Butcher1987,Hairer1993}.
The tree structure records compositions of derivatives of the vector
fields, while the associated coefficients distinguish the exact
solution from a numerical approximation. B-series also organise
stochastic Runge-Kutta expansions~\cite{Burrage2000}.
In rough path theory, tree-indexed expansions are developed through
branched rough paths \cite{Gub2010,HK2015}. Runge-Kutta methods
for Gaussian rough Hamiltonian systems are studied in
\cite{HHW2018}, and Redmann and Riedel~\cite{RR2022} develop
B-series expansions and order conditions for rough differential
equations.

Our construction extends this classical algebraic organisation to
controlled drivers. The rooted trees and elementary differentials
are retained, but the exact coefficients are generated recursively
by controlled rough integration against $\mathbf Z$. Consequently,
the extension requires more than a formal replacement of symbols:
the coefficient paths must remain controlled under the recursion,
and their regularity estimates must be compatible with their tree
degrees. These estimates are needed to control the remainders in the
exact and numerical expansions and to justify coefficient matching.
For the time driver $Z_t=t$, with the usual step-scaled tableau,
the construction recovers the classical B-series and Runge-Kutta
setting.

A separate issue is implementation. The controlled coefficients
contain iterated-integral information that may not be directly
available from sampled values of $Z$. Motivated by simplified rough
numerical schemes \cite{DNT2012,RR2022}, we therefore replace $Z$
by its piecewise linear interpolation. Canonical lifting of each
linear segment makes the tree coefficients explicit in terms of
increments. This produces a method using a fixed Runge-Kutta tableau,
but it also changes the reference equation. The numerical error
relative to the interpolated problem must therefore be distinguished
from the error of approximating the original CRDE. The latter is
handled through a separate Wong--Zakai assumption, motivated by
Gaussian rough path approximation theory~\cite{FR2014}.

\subsection{Main contributions and outline}
\label{subsec:intro-main-results}

The main contributions connect controlled integration, B-series
expansions, and numerical convergence. Let $p\in\NN$ satisfy
$(p+1)\alpha>1$. The statements summarised below are subject to
the smoothness and, where applicable, stage solvability, flow
stability, and boundedness hypotheses of the cited results.

\smallskip
\noindent\textbf{Controlled B-series for CRDE solutions.}
We construct controlled tree and forest coefficients by recursive
controlled rough integration in
Definition~\ref{def:controlled-forest-coefficients}, and combine
them with the classical elementary differentials to define controlled
B-series in Definition~\ref{def:controlled-b-series}.
Proposition~\ref{prop:well-posedness-controlled-coefficients}
establishes the controlled-path regularity and degree-dependent
estimates required for this construction.
Theorem~\ref{thm:exact-controlled-bseries} then gives the exact
solution expansion \eqref{eq:exact-controlled-bseries-expansion},
with a remainder of order $(p+1)\alpha$. This identifies the driving
coefficients relevant to local approximation while retaining the
classical separation between vector-field and driving data.

\smallskip
\noindent\textbf{Controlled Runge-Kutta methods and order conditions.}
Definition~\ref{def:controlled-rk} introduces controlled Runge-Kutta
(CRK) methods whose stepwise matrices and weights may depend on the
controlled coefficients. The numerical stage and output expansions in
Theorem~\ref{thm:rk-bseries} have the same elementary-differential
structure as the exact solution expansion. The tree matching
condition \eqref{eq:tree-order-condition} yields local accuracy
$O(|v-u|^{(p+1)\alpha})$ in
Theorem~\ref{thm:local-truncation-error-rk}. Under the flow stability
and boundedness hypotheses, Theorem~\ref{thm:global-controlled-rk}
gives global accuracy $O(h^{(p+1)\alpha-1})$, where $h$ is the mesh
size. Proposition~\ref{prop:tree-order-conditions-order-three}
makes the matching conditions explicit through tree order three.

\smallskip
\noindent\textbf{An increment-based simplification and classical tableaux.}
We introduce the simplified controlled Runge-Kutta (SCRK) method in
Definition~\ref{def:simplified-controlled-rk}.
Proposition~\ref{prop:tree-coefficients-piecewise-linear} gives
the explicit segmentwise coefficient formula
\eqref{eq:piecewise-linear-tree-coefficient}, and
Proposition~\ref{prop:simplified-factorization} gives the
corresponding factorization of the numerical coefficients. This
reduces the comparison to scalar weights determined by the tableau.
In particular, Proposition~\ref{prop:simplified-third-order-conditions}
shows that tree order three is equivalent to the classical third-order
Runge-Kutta conditions. Classical tableaux can therefore be used
without direct evaluation of controlled iterated integrals. The local
error estimate in
Theorem~\ref{thm:simplified-local-error-piecewise-linear} concerns
the solution of the piecewise linear problem, rather than the original
CRDE solution.

\smallskip
\noindent\textbf{Global convergence with two distinct error sources.}
Theorem~\ref{thm:global-simplified-smooth} bounds the SCRK
discretization error relative to the interpolated problem. Under
Assumption~\ref{ass:controlled-wong-zakai}, the solution of that
problem approximates the original CRDE solution at rate $r_0>0$.
Theorem~\ref{thm:global-simplified-rk} combines these estimates
and gives the global rate
$\min\{r_0,(p+1)\alpha-1\}$; the full error bounds are stated in
\eqref{eq:global-error-simplified} and
\eqref{eq:global-error-simplified-min}. For tree order three,
Corollary~\ref{cor:global-third-order} yields the rate
$\min\{r_0,4\alpha-1\}$. Thus, classical tableau order and
convergence order in the rough time scale play different roles:
increasing the former improves the discretization bound, but need
not improve the overall rate when driver approximation is the
limiting factor. Numerical experiments with fractional Brownian
controlled drivers illustrate this distinction.

\smallskip
\noindent\textbf{Outline of the paper.}
Section~\ref{sec:preliminaries} recalls controlled rough paths and
integration and states the well-posedness results in
Lemmas~\ref{thm:local-well-posedness}
and~\ref{cor:global-well-posedness}.
Section~\ref{sec:algebraic-framework} develops controlled B-series,
including the coefficient estimates in
Proposition~\ref{prop:well-posedness-controlled-coefficients} and
the solution expansion in Theorem~\ref{thm:exact-controlled-bseries}.
Section~\ref{sec:runge-kutta-methods} derives the numerical expansion
and local error bound in Theorems~\ref{thm:rk-bseries}
and~\ref{thm:local-truncation-error-rk}.
Section~\ref{sec:simplified-rk} develops the increment-based method,
establishes its local error estimate in
Theorem~\ref{thm:simplified-local-error-piecewise-linear}, and
identifies its third-order conditions in
Proposition~\ref{prop:simplified-third-order-conditions}.
Section~\ref{sec:global-error} establishes the global bounds in
Theorems~\ref{thm:global-controlled-rk},
\ref{thm:global-simplified-smooth}, and~\ref{thm:global-simplified-rk}.
The numerical experiments in
Subsection~\ref{sec:numerical-experiments} are reported in
Figure~\ref{fig:controlled-driver-mean-error} and
Table~\ref{tab:controlled-driver-rates}.

\smallskip
\noindent\textbf{Notation.}
Throughout the paper, $\alpha\in(1/3,1/2]$ is fixed, $\NN$ denotes
the set of positive integers, and $\RR$ denotes the real numbers.
We write $\mathcal L(\mathbb R^d,\mathbb R^m)$ for the space of
linear maps from $\mathbb R^d$ to $\mathbb R^m$ and
$X_{s,t}:=X_t-X_s$ for path increments.
The space of $\alpha$-H\"older paths is denoted by
$C^\alpha([0,T];\mathbb R^d)$, with H\"older seminorm
$\|\cdot\|_{\alpha;[0,T]}$. Two-parameter increments are defined
on $\Delta_T:=\{(s,t)\in[0,T]^2:s\le t\}$.

\section{Controlled-driven  rough differential equations}\label{sec:preliminaries}
This section recalls the basic framework of controlled rough paths, controlled rough integration, and controlled-driven rough differential equations, together with the well-posedness results needed for the controlled Runge-Kutta methods developed below.

\begin{definition}\cite{JDT2022, FH2020}\label{def:rough-path}
An \emph{$\alpha$-H\"older rough path} over $\mathbb{R}^d$ is a pair
$
\mathbf{X} = (X, \mathbb{X})$ with $X \in C^{\alpha}([0, T];\mathbb{R}^d)$ and $\mathbb X:\Delta_T\to\mathbb R^{d\times d}$
satisfying the Chen identity:
\begin{equation*} 
\mathbb{X}_{s,t} = \mathbb{X}_{s,u} + \mathbb{X}_{u,t} + X_{s,u} \otimes X_{u,t},\qquad \forall 0 \le s \le u \le t \le T.
\end{equation*}
The set of all such $\mathbf{X}$ is denoted by $\mathcal{C}^\alpha([0,T];\mathbb{R}^d)$.
We define the norm of a rough path by
$$
\|\mathbf{X}\|_{\alpha;[0,T]} := \|X\|_{\alpha;[0,T]} + \|\mathbb{X}\|_{2\alpha;[0,T]}, $$
where 
$$\|X\|_{\alpha ;[0, T]}=\sup _{0 \leq  s<t \leq T} \frac{\left|X_{s , t}\right|}{|t-s|^\alpha},\qquad \|\mathbb{X}\|_{2 \alpha ;[0, T]}=\sup _{0 \leq s<t \leq T} \frac{\left|\mathbb{X}_{s, t}\right|}{|t-s|^{2 \alpha}}.$$
\end{definition}

\begin{definition}\cite{Gubinelli2004}\label{def:controlled-path}
Let $\mathbf{X} = (X,\mathbb{X}) \in \mathcal{C}^\alpha([0,T];\mathbb{R}^d)$.
For \(Y \in C^\alpha([0,T];\mathbb{R}^e)\) and \(Y'\in C^\alpha([0,T];\mathbb{R}^{e\times d})\), the pair $\mathbf{Y}:= (Y, Y')$ is called an \emph{$\bX$-controlled rough path} if $\|R^Y\|_{2\alpha;[0,T]} <\infty$, where 
the remainder $R^Y : \Delta_T \to \mathbb{R}^e $ is given by 
\begin{equation*}
Y_{s,t} = Y'_s X_{s,t} + R^Y_{s,t}, \qquad 0 \le s \le t \le T.
\end{equation*}
\end{definition}

In this case, $Y$ is said to be \emph{controlled by $X$}, and $Y'$ is called the \emph{Gubinelli derivative} of $Y$ with respect to $X$.
The space of all $\bX$-controlled rough paths $\mathbf{Y}$ taking values in $\mathbb{R}^e$ is denoted by
\(
\mathcal{D}_\bX^\alpha([0,T];\mathbb{R}^e),
\)
and is equipped with the norm \cite{A23}  
\begin{equation}\label{eq:controlled-norm}
\|(Y,Y')\|_{X,\alpha;[0,T]}
:=
|Y_0| + |Y'_0| + \|Y'\|_{\alpha;[0,T]} + \|R^Y\|_{2\alpha;[0,T]},
\end{equation}
where $|Y_0|$ denotes the Euclidean norm of $Y_0\in\mathbb R^e$
and $|Y'_0|$ denotes the operator norm of
$Y'_0\in\mathcal L(\mathbb R^d,\mathbb R^e)$. The term $\left|Y_0^{\prime}\right|$ is included to ensure that the controlled path norm also bounds the Gubinelli derivative $Y^{\prime}$ on $[0, T]$.

For
$A\in \mathcal{L}(\mathbb{R}^d,\mathcal{L}(\mathbb{R}^{m},\mathbb{R}^e))$ and
$B\in \mathcal{L}(\mathbb{R}^d,\mathbb{R}^{m}),$
we use the notation~\cite[Remark 4.12]{FH2020}
\begin{equation}\label{E2-3}
(A\diamond B)(v\otimes \tilde v):=A(v)\bigl(B\tilde v\bigr),
\qquad  v,\tilde v\in \mathbb{R}^d.
\end{equation}
for the contraction appearing in the controlled rough integral.

\begin{lemma} \cite{FH2020, Gubinelli2004}
\label{thm:controlledintegral}
Let $\bX\in \mathcal{C}^\alpha([0,T];\mathbb{R}^d),$ 
$
(H,H') \in \mathcal{D}_\bX^\alpha([0,T];\mathcal{L}(\mathbb{R}^{m},\mathbb{R}^{e}))$
and
$
\bZ=(Z,Z') \in \mathcal{D}_\bX^\alpha([0,T];\mathbb{R}^{m}).
$
Then, for every $0\le s\le t\le T$, the limit 
\begin{equation}\label{eq:controlled-integral-definition}
\sint_s^t H_r\,d\bZ_r
:=
\lim_{|P|\to 0}
\sum_{[u,v]\in P}
\left(
H_u Z_{u,v}
+
(H'_u \diamond Z'_u)\mathbb{X}_{u,v}
\right) \in \mathbb{R}^{e}
\end{equation}
exists and is independent of the partition $P$ of $[s,t]$.
Moreover, there exists a constant $C_{\alpha,T}>0$ such that
\begin{align*}
&\left|
\sint_s^t H_r\,d\bZ_r
-
H_s Z_{s,t}
-
(H'_s\diamond Z'_s)\mathbb{X}_{s,t}
\right|\\ \nonumber
&\ \ \le
C_{\alpha,T}
\bigl(1+\|\mathbf{X}\|_{\alpha;[0,T]}\bigr)^2
\|(H,H')\|_{X,\alpha;[0,T]}
\|(Z,Z')\|_{X,\alpha;[0,T]}
|t-s|^{3\alpha}.
\end{align*}
We call $\sint_s^t H_r\,d\bZ_r$  the {\it controlled rough integral of $H$ against $\bZ$.}
\end{lemma}

The above result shows that the controlled rough integral admits the local expansion
\[
\sint_s^t H_r\,d\bZ_r
=
H_s Z_{s,t}
+
(H'_s\diamond Z'_s)\mathbb X_{s,t}
+
O(|t-s|^{3\alpha}).
\]
This expansion underlies the numerical methods developed in Sections~\ref{sec:runge-kutta-methods} and~\ref{sec:simplified-rk}.

\begin{lemma} \cite{LiGao2026}
\label{prop:integral-controlled}
Let
$
I_t:=\sint_0^t H_r\,d\bZ_r$ with
$ 0\le t\le T.$
Then
\(
(I,I')\in \mathcal{D}_\bX^\alpha([0,T];\mathbb{R}^e)\), where \(I'_t:=H_tZ'_t\) and the product $H_tZ'_t$ denotes the composition of linear maps, namely
\(
(H_tZ'_t)v:=H_t(Z'_t v)\)
for \(v\in \mathbb{R}^d.
\)
\end{lemma}

\label{subsec:controlled-driven-rdes}
 For $k\in \mathbb{Z}_+$, we let
$C_b^k(
\mathbb R^e;
\mathcal L(\mathbb{R}^{m},\mathbb R^e)
)$ be the space of $k$ times continuously
Fr\'echet differentiable maps $F:\mathbb{R}^e
\to\mathcal L(\mathbb{R}^{m},\mathbb R^e)$ satisfying
\[
\sup_{x\in \mathbb{R}^e}
\|D^jF(x)\|_{\mathcal L^j(\mathbb R^e;
\mathcal L(\mathbb{R}^{m},\mathbb R^e)
)}
<\infty,
\quad j=0,\ldots,k,
\]
with the convention that $$D^0F:=F, \qquad \mathcal L^0(\mathbb R^{e},\mathcal L(\mathbb{R}^{m},\mathbb R^e)
):=\mathcal L(\mathbb{R}^{m},\mathbb R^e).$$

Controlled paths are stable under smooth composition.
 
\begin{lemma} \cite{Gubinelli2004,LiGao2026}
\label{lem:composition}
Let
\((Y,Y')\in\mathcal D_\bX^\alpha([0,T];\mathbb R^e)\)
and
\(F\in C_b^2\bigl(
\mathbb R^e;
\mathcal L(\mathbb{R}^{m},\mathbb R^e)
\bigr)\).
Then $\bigl(F(Y),(F(Y))'\bigr)\in\mathcal D_\bX^\alpha
\bigl([0,T];
\mathcal L(\mathbb{R}^{m},\mathbb R^e)
\bigr)$ is an $\bX$-controlled rough path, where 
\begin{align*}
F(Y):[0,T] \rightarrow&  \mathcal L(\mathbb{R}^{m},\mathbb R^e), \qquad t\mapsto   F(Y_t)\\
(F(Y))':[0,T] \rightarrow& \mathcal L(\mathbb{R}^{d}, \mathcal L(\mathbb{R}^{m},\mathbb R^e)), \qquad t\mapsto DF(Y_t)Y'_t.
\end{align*}
\end{lemma}

We now introduce the controlled-driven rough differential equation. Unlike the standard rough differential equation driven directly by \(X\), the equation considered here is driven by a path \(Z\) controlled by \(X\). 

\begin{definition} \cite{LiGao2026} \label{def:controlled-driven-rde} Let $ (Z,Z')\in \mathcal{D}_\bX^\alpha([0,T];\mathbb{R}^{m})$
and
$
F\in C_b^2\bigl(\mathbb{R}^e;\mathcal{L}(\mathbb{R}^{m},\mathbb{R}^e)\bigr). $
For \(0<T_0\le T\), a pair $ (Y,Y')\in \mathcal{D}_\bX^\alpha([0,T_0];\mathbb{R}^e) $ is called \emph{a solution of the controlled-driven rough differential equation} 
\begin{equation}\label{eq:controlled-driven-rde} 
dY_t=F(Y_t)\,d\bZ_t,\quad Y_0=y_0\in\mathbb{R}^e,
\end{equation} 
on \([0,T_0]\) if it satisfies, for every \(0\le t\le T_0\),
\begin{equation}\label{eq:controlled-driven-fixed-point} 
Y_t = y_0+\sint_0^t F(Y_r)\,d\bZ_r, \qquad Y'_t = F(Y_t)Z'_t. 
\end{equation} 
\end{definition}

\begin{remark}
\begin{enumerate}
\item[\rm(i)]
If \(m=d\), \(Z_t=X_t\)  and \(Z'_t=\operatorname{Id}_{\mathbb R^d}\),
then \((Z,Z')=(X,\operatorname{Id}_{\mathbb R^d})\) and
\eqref{eq:controlled-driven-rde} reduces to the usual rough differential
equation
\[
dY_t=F(Y_t)d\bX_t,\qquad Y_0=y_0 .
\]
Moreover, \(Y'_t=F(Y_t)Z'_t=F(Y_t)\), which is the Gubinelli derivative of the path $Y$.

\item[\rm(ii)]
The integral in \eqref{eq:controlled-driven-fixed-point} is understood in
the sense of Lemma~\ref{thm:controlledintegral}. By Lemmas~
\ref{prop:integral-controlled} and \ref{lem:composition},
\(F(Y)\) is an admissible integrand with respect to \(Z\), and
\[
t\mapsto \sint_0^tF(Y_r)d\bZ_r
\]
is controlled by \(X\) with the Gubinelli derivative \(F(Y_t)Z'_t\).
\end{enumerate}
\end{remark}

We now recall the basic well-posedness results of~\meqref{eq:controlled-driven-rde}.

\begin{lemma} \cite[Theorem~3.9]{LiGao2026} \label{thm:local-well-posedness}
Let
\(
(Z,Z')\in \mathcal{D}_\bX^\alpha([0,T];\mathbb{R}^{m})\) and \(F\in C_b^3\bigl(\mathbb{R}^e;\mathcal{L}(\mathbb{R}^{m},\mathbb{R}^e)\bigr).
\) Then there exists \(\tau\in(0,T]\) such that the controlled-driven rough differential equation
\[
dY_t=F(Y_t)\,d\bZ_t,
\quad 
Y_0=y_0\in \mathbb{R}^e,
\]
admits a unique local solution
\(
(Y,Y')\in \mathcal{D}_\bX^\alpha([0,\tau];\mathbb{R}^e).
\)
Moreover, for any $0\le t\le \tau$,
\begin{equation}\label{eq:solution-gubinelli-derivative}
Y'_t=F(Y_t)Z'_t.
\end{equation}
\end{lemma}

\begin{lemma} \cite[Theorem~3.11]{LiGao2026}
\label{cor:global-well-posedness}
The local solution as in Lemma~\ref{thm:local-well-posedness} extends uniquely to a solution
\(
(Y,Y')\in \mathcal{D}_\bX^\alpha([0,T];\mathbb{R}^e)
\)
on the whole interval $[0,T]$. Moreover, for any $0\le t\le T$,
\(
Y'_t=F(Y_t)Z'_t\).
\end{lemma}

\section{Controlled B-series expansions}\label{sec:algebraic-framework}
This section develops the controlled B-series framework used in the analysis of controlled Runge-Kutta methods. We introduce the rooted-tree structure, elementary differentials, and controlled coefficients, and then use them to derive the exact controlled B-series expansion of the solution of the controlled-driven rough differential equation.

We begin with the underlying tree and forest notation. Let $\mathcal{A}:=\{1,\ldots,m\}$ be the set of decorations, and let $\mathcal{T}$ denote the set of rooted trees whose vertices are decorated by elements of $\mathcal{A}$. Let $\mathcal{F}^0$ be the free commutative monoid generated by $\mathcal{T}$, with unit $\etree$ representing the empty forest. We set \[ \mathcal{T}^0:=\mathcal{T}\sqcup\{\etree\}, \qquad \mathcal{F}:=\mathcal{F}^0\setminus\{\etree\}. \]
For a forest
\(
\eta=\tau_1\cdots\tau_n\in\mathcal F^0
\)
and a decoration \(a\in \mathcal A\), define
$
[\eta]_a=[\tau_1\cdots\tau_n]_a
$
to be the rooted tree obtained by adding a new root decorated by \(a\)
and connecting it to the roots of \(\tau_1,\ldots,\tau_n\). In
particular,
$
[\etree]_a=\bullet_a
$
is the one-vertex tree with decoration \(a\).
The order of a forest is the total number of vertices, defined
recursively by
\begin{equation}\label{C3-1}
|\etree|:=0,
\qquad
|\tau_1\cdots\tau_n|
:=
\sum_{i=1}^n|\tau_i|,
\qquad
|[\eta]_a|
:=
|\eta|+1.
\end{equation}
For a truncation level \(N\in\mathbb N\), define
\[
\mathcal T_{\leq N}
:=
\{\tau\in\mathcal T:|\tau|\leq N\},
\qquad
\mathcal F_{\leq N}
:=
\{\eta\in\mathcal F:|\eta|\leq N\}.
\]

We next recall the symmetry factor and the tree factorial, which are
standard in B-series theory and the tree calculus for Runge-Kutta
methods.

\begin{definition}\cite{Butcher1963,RR2022}
\label{def:symmetry-factor}
The \emph{symmetry factor} $\sigma:\mathcal T^0\to\mathbb N$ and the \emph{tree factorial} $\gamma:\mathcal T^0\to\mathbb N$ are recursively defined by
\begin{align*}
\sigma(\etree):=&\ \sigma(\bullet_a):=1,
\qquad
\sigma(\tau):=
\prod_{j=1}^{\ell}
m_j!\sigma(\rho_j)^{m_j},\\
\gamma(\etree):=&\ \gamma(\bullet_a):=1,
\qquad
\gamma(\tau)
:=
|\tau|
\prod_{i=1}^{k}\gamma(\tau_i),
\end{align*}
where
\[
\tau=[\tau_1\cdots\tau_k]_a
=
[\rho_1^{m_1}\cdots\rho_\ell^{m_\ell}]_a,
\qquad a\in\mathcal A,
\]
with \(\rho_1,\ldots,\rho_\ell\) distinct and \(m_1,\ldots,m_\ell\) their
multiplicities.  
\end{definition}

Recall in~\eqref{eq:controlled-driven-rde} that $Y$ takes values in \(\mathbb R^e\) and that the controlled driver is
\[
Z=(Z^1,\ldots,Z^m):[0,T]\to\mathbb R^m.
\]
Let
\(
F:\mathbb R^e\to\mathcal L(\mathbb R^m,\mathbb R^e),
\)
and let $e_1,\ldots,e_m$ denote the canonical basis of $\mathbb R^m$.
For each $y\in\mathbb R^e$, define
\begin{equation}\label{eq:component-vector-fields}
F_a(y):=F(y)e_a,
\qquad a\in\mathcal A.
\end{equation}
Then \(F_a:\mathbb R^e\to\mathbb R^e\) is the vector field associated with the
\(a\)-th driving channel. Since
$
dZ_t=\sum_{a=1}^m e_a\,dZ_t^a,
$
the rough differential equation
\[
dY_t=F(Y_t)\,d\bZ_t,
\qquad
Y_0=y_0,
\]
is equivalently written in component form as
\begin{equation}\label{eq:component-controlled-rde}
dY_t
=
\sum_{a=1}^m F_a(Y_t)\,d\bZ_t^a,
\qquad
Y_0=y_0.
\end{equation}

Next, we recall the elementary differentials associated with the vector fields
\(F_1,\ldots,F_m\).

\begin{definition}\cite{RR2022}\label{def:elementary-differentials}
Let
$
F_1,\ldots,F_m:\mathbb{R}^e\to\mathbb{R}^e
$  
be the vector fields as in \eqref{eq:component-vector-fields}.
For $a\in\mathcal A$,
\(y\in\mathbb{R}^e\) and \(\tau\in\mathcal{T}^0\), the elementary differential
\(F(\tau):\mathbb{R}^e\to\mathbb{R}^e\) is defined  by
\begin{equation*}
F(\etree)(y):=y,
\quad
F(\bullet_a)(y):=F_a(y),
\end{equation*}
and, for the tree $
\tau=[\tau_1\cdots\tau_n]_a,
$  $ n\ge 1,$
\begin{equation*}
F(\tau)(y)
:=
D^nF_a(y)
\bigl[
F(\tau_1)(y),\ldots,F(\tau_n)(y)
\bigr].
\end{equation*}
\end{definition}

We now define the coefficients recursively by controlled rough
integration with respect to the components of \(Z\). These coefficients
play the role of the tree coordinates used in branched rough paths
\cite{Gub2010,HK2015}.

\begin{definition}
\label{def:controlled-forest-coefficients}
Let $
\mathbf Z=(Z,Z')
\in \mathcal D_\bX^\alpha([0,T];\mathbb R^m) 
$. Fix \(s\in[0,T]\) and define the coefficient path
\[
Z^\eta_{s,\bigcdot}:[s,T]\rightarrow\mathbb R,
\qquad
t\mapsto Z^\eta_{s,t},
\]
recursively as follows.

\begin{enumerate}
\item[\rm(i)]
For the empty forest $\etree$, set 
$$
Z^\etree_{s,t}:=1\in \mathbb{R}.
$$

\item[\rm(ii)] 
If \(\eta=\tau_1\cdots\tau_n\in\mathcal F\) is a nonempty forest, set
$$
 Z^\eta_{s,t}
:=
\prod_{j=1}^n  Z^{\tau_j}_{s,t}\in \mathbb{R} .
$$

\item[\rm(iii)] If \(\tau=[\eta]_a\) is the tree obtained by grafting a forest
\(\eta\in\mathcal F^0\) onto a new root labelled by \(a\in\mathcal A\), set
\begin{equation}\label{eq:controlled-coefficient-tree}
 Z^\tau_{s,t}
:=
\sint_s^t  Z^\eta_{s,u}\,d\bZ_u^a \in \mathbb{R} .
\end{equation}
\end{enumerate}

The integral in \eqref{eq:controlled-coefficient-tree} is understood as the controlled rough integral in \meqref{eq:controlled-integral-definition}. 
\end{definition}

\begin{proposition}
\label{prop:well-posedness-controlled-coefficients}
Let
$
(Z,Z')
\in
\mathcal{D}_\bX^\alpha([0,T];\mathbb{R}^m).
$
Then, for any \(p\in\mathbb N\), there exists a constant
\(C_{p,T}>0\) such that, for $s \leq r \leq t \leq T$ and 
\(\eta\in\mathcal{F}_{\le p}\), the path $\mathbf Z^\eta_{s,\bigcdot} = (Z^\eta_{s,\bigcdot}, (Z^\eta_{s,\bigcdot})')$ is \(\bX\)-controlled  and satisfies  
\begin{equation}\label{eq:controlled-coefficient-size}
\begin{gathered}
| Z_{s,r}^{\eta}|
\le
C_{p,T}|r-s|^{|\eta|\alpha},
\qquad 
|(Z_{s,\bigcdot}^{\eta})'_r|
\le
C_{p,T}|r-s|^{(|\eta|-1)\alpha},
\\[2mm]
\|( Z_{s,\bigcdot}^{\eta})'\|_{\alpha;[s,t]}
+
\|R^{Z_{s,\bigcdot}^{\eta}}\|_{2\alpha;[s,t]}
\le
C_{p,T}|t-s|^{(|\eta|-2)_+\alpha}.
\end{gathered}
\end{equation}
where $$R_{u, v}^{Z_{s, \bigcdot}^\eta}:=Z_{s, v}^\eta-Z_{s, u}^\eta-\left(Z_{s, \bigcdot}^\eta\right)_u^{\prime} X_{u, v}.$$
\end{proposition}

The proof of Proposition~\ref{prop:well-posedness-controlled-coefficients} needs the following technical lemma.

\begin{lemma}
\label{lem:controlled-remainder-integration}
Given \(0\le s<t\le T\), let $(R^s,(R^s)')\in
\mathcal D_\bX^\alpha([s,t];\mathbb R^e).$ Suppose that
\begin{enumerate}
\item[\rm(i)]
\(
R_s^s=0\), \(|R_r^s|
\le
M|r-s|^\beta\) and \(
|(R^s)'_r|
\le
M|r-s|^{\beta-\alpha},\,\beta \geq \alpha
\)
for every \(r\in[s,t]\), where  \(M\in \RR_{>0}\) is a positive constant, independent of \(s\) and \(t\).
\item[\rm(ii)]
\(
\|(R^s)'\|_{\alpha;[s,t]}
+
\|\mathcal R^{R^s}\|_{2\alpha;[s,t]}
\le
M|t-s|^{(\beta-2\alpha)_+}
\)
for every \(r\in[s,t]\), where  \(M\) is a positive constant, independent of \(s\) and \(t\), and
\[
\mathcal R^{R^s}_{u,v}
:=
R^s_{u,v}-(R^s)'_uX_{u,v},
\qquad
(\beta-2\alpha)_+
:=
\max\{\beta-2\alpha,0\}.
\]
\end{enumerate}
Then there exists a constant \(C>0\), depending only on
\(\alpha\), \(T\), \(X\) and \((Z,Z')\), such that
\[
\left|
\sint_s^t R_r^s\,d\bZ_r^a
\right|
\le
CM|t-s|^{\beta+\alpha},\quad a\in\mathcal A.
\]
\end{lemma}

\begin{proof}
By \eqref{eq:controlled-integral-definition}, we have
\[
\sint_s^t R_r^s\,d\bZ_r^a
=
R_s^sZ_{s,t}^a
+
\bigl((R^s)'_s\diamond(Z^a)'_s\bigr)\mathbb X_{s,t}
+
\mathcal E_{s,t},
\]
where
$
|\mathcal E_{s,t}|
\le
C
\|(R^s,(R^s)')\|_{X,\alpha;[s,t]}
|t-s|^{3\alpha},
$
and $\diamond$ is the contraction defined in \eqref{E2-3}.
Since \(R_s^s=0\), the first term vanishes.
For the second term, if \(\beta>\alpha\), then
$
(R^s)'_s=0,
$
while if \(\beta=\alpha\), then
$
|(R^s)'_s|\le M.
$
Therefore, in both cases,
\[
\left|
\bigl((R^s)'_s\diamond(Z^a)'_s\bigr)
\mathbb X_{s,t}
\right|
\le
CM|t-s|^{\beta+\alpha}.
\]

It remains to estimate the remainder term. By the assumptions $R_s^s=0$ and $|(R^s)'_s|\le M$, the controlled norm on $[s, t]$ satisfies
$$
\left\|\left(R^s,\left(R^s\right)^{\prime}\right)\right\|_{X, \alpha ;[s, t]}=\left|R_s^s\right|+\left|\left(R^s\right)_s^{\prime}\right|+\left\|\left(R^s\right)^{\prime}\right\|_{\alpha ;[s, t]}+\left\|\mathcal R^{R^s}\right\|_{2 \alpha ;[s, t]} \leq CM|t-s|^{(\beta-2 \alpha)_{+}} .
$$
Hence, 
\[
|\mathcal E_{s,t}|
\le
CM|t-s|^{3\alpha+(\beta-2\alpha)_+}.
\]
If \(\beta\ge2\alpha\), then
$
3\alpha+(\beta-2\alpha)_+=\beta+\alpha.
$
If \(\alpha\le\beta<2\alpha\), then
$
3\alpha+(\beta-2\alpha)_+=3\alpha>\beta+\alpha,
$
therefore
$$
|t-s|^{3\alpha}
\le
T^{2\alpha-\beta}|t-s|^{\beta+\alpha}.
$$
Absorbing the factor \(T^{2\alpha-\beta}\) into the constant gives
$
|\mathcal E_{s,t}|
\le
CM|t-s|^{\beta+\alpha}.
$

Combining the above estimates yields
$$
\left|
\sint_s^t R_r^s\,d\bZ_r^a
\right|
\le
CM|t-s|^{\beta+\alpha}.
$$
This completes the proof of Lemma~\ref{lem:controlled-remainder-integration}.
\end{proof}

\begin{proof}[Proof of Proposition \ref{prop:well-posedness-controlled-coefficients}]
We use induction on the order $|\eta|\geq 0$. For every forest \(\eta\), we prove that the pair
$
\mathbf Z^\eta_{s,\bigcdot}
=
\left(
Z^\eta_{s,\bigcdot},
(Z^\eta_{s,\bigcdot})'
\right)
$
is \(\bX\)-controlled on \([s,T]\) and satisfies
\eqref{eq:controlled-coefficient-size}
for all \(s\le r\le t\le T\).

{\bf (Initial step).}
For the empty forest \(\etree\), by Definition~\ref{def:controlled-forest-coefficients}(i), we have  
$
Z^\etree_{s,t}=1.
$
Hence
$$
\mathbf Z^\etree_{s,\bigcdot}
=
\left(
Z^\etree_{s,\bigcdot},
(Z^\etree_{s,\bigcdot})'
\right)
=
(1,0)
,$$ 
where \(1\) denotes the constant path and \(0\) denotes the zero Gubinelli derivative.
Therefore, \(\mathbf Z^\etree_{s,\bigcdot}\) is a constant \(\bX\)-controlled path with
zero controlled remainder.

We next verify the case of forests of order one.  Every forest of order one is a one-node tree $\bullet_a=[\etree]_a$. By Definition~\ref{def:controlled-forest-coefficients},
$
 Z^{\bullet_a}_{s,r}
=
\sint_s^r 1\,d\bZ_u^a
=
Z^a_{s,r}.
$
Since $\left(Z, Z^{\prime}\right) \in \mathcal{D}_\bX^\alpha\left([0, T] ; \mathbb{R}^m\right)$, the component $Z^a$ is $X$-controlled and $\alpha$-H\"older continuous. Therefore,
$$
|Z_{s, r}^{\bullet_a}|=|Z_{s, r}^a| \leq C|r-s|^\alpha=C|r-s|^{|\bullet_a|\alpha}.
$$
Moreover,
$
Z_{s, s}^{{\bullet}_a}=0$ and 
$$
|(Z_{s, \bigcdot}^{{\bullet}_a})_r^{\prime}|=|(Z^a)_r^{\prime}| \leq|(Z^a)_0^{\prime}|+\|(Z^a)^{\prime}\|_{\alpha ;[0, T]} r^\alpha \leq C=C|r-s|^{(|\bullet_a|-1) \alpha}, \quad r \in[s, T],
$$
where $C$ is independent of $s, r$ and $t$. For $s \leq u \leq v \leq T$, we have 
$$
R_{u, v}^{Z_{s, \bigcdot}^{\bullet_a}}  =(Z_{s, \bigcdot}^{\bullet_a})_{u, v}-(Z_{s, \bigcdot}^{\bullet_a})_u^{\prime} X_{u, v} 
 =Z_{u, v}^a-(Z^a)_u^{\prime} X_{u, v} 
 =R_{u, v}^{Z^a},
$$
and so
\[
\begin{aligned}
&\|(Z_{s,\bigcdot}^{\bullet_a})'\|_{\alpha;[s,t]}
+\|R^{Z_{s,\bigcdot}^{\bullet_a}}\|_{2\alpha;[s,t]}
\leq
\|(Z^a)'\|_{\alpha;[0,T]}
+\|R^{Z^a}\|_{2\alpha;[0,T]}
\leq C ,
\end{aligned}
\]
which proves the required estimate since \(|\bullet_a|=1\).

{\bf (Inductive step).}
Assume that the conclusions of the proposition hold for every nonempty forest of order at most $k$ for some $k\geq 1$. We prove them for forests of order $k+1$. The remaining proof is decomposed into the following two steps.

\noindent{\bf Step 1.} Suppose that $\eta$ consists of at least two trees
$$
\eta=\tau_1 \cdots \tau_n, \quad n \geq 2, \quad|\eta|=\left|\tau_1\right|+\cdots+\left|\tau_n\right|=k+1 .
$$
Since $n \geq 2$, each tree satisfies $\left|\tau_j\right| \leq k$. The induction hypothesis therefore applies to all paths
$
Z_{s, \bigcdot}^{\tau_j}:[s, T]\rightarrow \RR.
$ 
By Definition~\ref{def:controlled-forest-coefficients}, we obtain
$
Z_{s, r}^\eta=\prod_{j=1}^n Z_{s, r}^{\tau_j} .
$
Since finite products of $\bX$-controlled paths are again $\bX$-controlled, the path
$\mathbf Z^\eta_{s,\bigcdot}
$
is $\bX$-controlled on $[s, T]$.
By the induction hypothesis,
$$
|Z_{s, r}^\eta|  =\prod_{j=1}^n|Z_{s, r}^{\tau_j}|  \leq C \prod_{j=1}^n|r-s|^{|\tau_j| \alpha}  =C|r-s|^{|\eta| \alpha}.
$$
Moreover, every factor vanishes at $r=s$. Hence
$
Z_{s, s}^\eta=0 .
$
The product rule gives
$$
(Z_{s, \bigcdot}^\eta)_r^{\prime}=\sum_{j=1}^n(Z_{s,\bigcdot}^{\tau_j})_r^{\prime} \prod_{\substack{i=1 \\ i \neq j}}^n Z_{s, r}^{\tau_i} .
$$
Therefore,
$$
|(Z_{s,\bigcdot}^\eta)_r^{\prime}|  \leq \sum_{j=1}^n|(Z_{s,\bigcdot}^{\tau_j})_r^{\prime}| \prod_{i \neq j}|Z_{s, r}^{\tau_i}|  \leq C \sum_{j=1}^n|r-s|^{(|\tau_j|-1) \alpha} \prod_{i \neq j}|r-s|^{|\tau_i| \alpha} 
=C|r-s|^{(|\eta|-1) \alpha} .
$$
To estimate the Hölder seminorm of the Gubinelli derivative, let
\(s\le u<v\le t\). Using the product rule, we obtain
\[
(Z_{s,\bigcdot}^\eta)'_v-(Z_{s,\bigcdot}^\eta)'_u
=
\sum_{j=1}^{n}
\left(
(Z_{s,\bigcdot}^{\tau_j})'_v
\prod_{i\neq j}Z_{s,v}^{\tau_i}
-
(Z_{s,\bigcdot}^{\tau_j})'_u
\prod_{i\neq j}Z_{s,u}^{\tau_i}
\right).
\]

By the induction hypothesis and the product estimate for controlled paths, it follows that
\[
|(Z_{s,\bigcdot}^\eta)'_v-(Z_{s,\bigcdot}^\eta)'_u|
\leq
C|t-s|^{(|\eta|-2)\alpha}|v-u|^\alpha .
\]
Then 
\[
\|(Z_{s,\bigcdot}^\eta)'\|_{\alpha;[s,t]}
\leq
C|t-s|^{(|\eta|-2)\alpha}.
\]
In addition, for every \(j=1,\ldots,n\) and \(s\le u<v\le t\), the
controlled expansion and the induction hypothesis give
\[
|
( Z^{\tau_j}_{s,\bigcdot})_{u,v}
|
\le
|
( Z^{\tau_j}_{s,\bigcdot})'_u
|
|X_{u,v}|
+
|
R^{ Z^{\tau_j}_{s,\bigcdot}}_{u,v}
|
\le
C|t-s|^{(|\tau_j|-1)\alpha}|v-u|^\alpha
+
C|t-s|^{(|\tau_j|-2)_+\alpha}|v-u|^{2\alpha}
\le
C|t-s|^{(|\tau_j|-1)\alpha}|v-u|^\alpha.
\]
Hence
$$
\|
 Z^{\tau_j}_{s,\bigcdot}
\|_{\alpha;[s,t]}
\le
C|t-s|^{(|\tau_j|-1)\alpha}.
$$

It remains to estimate the controlled remainder. Expanding the product
at the left endpoint \(u\) gives
\[
R^{ Z^\eta_{s,\bigcdot}}_{u,v}
=
\sum_{j=1}^n
R^{Z^{\tau_j}_{s,\bigcdot}}_{u,v}
\prod_{i\ne j} Z^{\tau_i}_{s,u}
+
\sum_{\substack{J\subseteq\{1,\ldots,n\}\\ |J|\ge2}}
\left(
\prod_{j\in J}
(Z^{\tau_j}_{s,\bigcdot})_{u,v}
\right)
\left(
\prod_{i\notin J}
 Z^{\tau_i}_{s,u}
\right).
\]
For the first sum, the induction hypothesis gives
$$
|
R^{Z^{\tau_j}_{s,\bigcdot}}_{u,v}
\prod_{i\ne j} Z^{\tau_i}_{s,u}
|
\le
C|t-s|^{(|\eta|-2)\alpha}|v-u|^{2\alpha}.
$$
If $|J|=q\ge2$, then
$$
|
\prod_{j\in J}
( Z^{\tau_j}_{s,\bigcdot})_{u,v}
\prod_{i\notin J}
Z^{\tau_i}_{s,u}
|
\le
C|t-s|^{(|\eta|-q)\alpha}|v-u|^{q\alpha}\le
C|t-s|^{(|\eta|-2)\alpha}|v-u|^{2\alpha}.
$$
Consequently,
$$
\|
R^{Z^\eta_{s,\bigcdot}}
\|_{2\alpha;[s,t]}
\le
C|t-s|^{(|\eta|-2)\alpha}.
$$
Combining the last two estimates yields
\[
\|
( Z^\eta_{s,\bigcdot})'
\|_{\alpha;[s,t]}
+
\|
R^{ Z^\eta_{s,\bigcdot}}
\|_{2\alpha;[s,t]}
\le
C|t-s|^{(|\eta|-2)_+\alpha}.
\]
Hence all the stated estimates hold for forests of order \(k+1\)
consisting of at least two trees.

\noindent{\bf Step 2.} 
Consider a forest \(\eta\) of order \(k+1\ge2\)
consisting of a single tree \(\tau\). Write
$
\eta=\tau=[\eta']_a,
$
with \(\eta'\in\mathcal F^0\), \(a\in\mathcal A\) and
\(|\eta'|=k\).
Since \(|\eta'|=k\), the induction hypothesis applies to
\(Z_{s,\bigcdot}^{\eta^{\prime}}\). In particular,
\[
 Z^{\eta'}_{s,s}=0,
\quad
|Z^{\eta'}_{s,r}|
\le
C|r-s|^{|\eta'|\alpha},\quad
|
(Z^{\eta'}_{s,\bigcdot})'_r
|
\le
C|r-s|^{(|\eta'|-1)\alpha}
\]
and
\[
\|
( Z^{\eta'}_{s,\bigcdot})'
\|_{\alpha;[s,t]}
+
\|
R^{ Z^{\eta'}_{s,\bigcdot}}
\|_{2\alpha;[s,t]}
\le
C|t-s|^{(|\eta'|-2)_+\alpha}.
\]

By Definition~\ref{def:controlled-forest-coefficients}, we have
$ Z^\tau_{s,r}
=
\sint_s^r
 Z^{\eta'}_{s,u}\,d\bZ_u^a.
$
Lemma~\ref{lem:controlled-remainder-integration} applies with 
$
R^s=Z^{\eta'}_{s,\bigcdot}$
and
$
\beta=|\eta'|\alpha.
$
Hence,
\[
| Z^\tau_{s,r}|=
\left|
\sint_s^r
Z^{\eta'}_{s,u}\,d\bZ_u^a
\right|
\le
C|r-s|^{(|\eta'|+1)\alpha}
=
C|r-s|^{|\tau|\alpha}.
\]
Moreover,
$
 Z^\tau_{s,s}=0.
$
By Lemma~\ref{prop:integral-controlled}, the path
$\mathbf Z^\tau_{s,\bigcdot}$ is \(\bX\)-controlled on \([s,T]\), and
its Gubinelli derivative is
$
\left(Z^\tau_{s,\bigcdot}\right)'_r
=
 Z^{\eta'}_{s,r}(Z^a)'_r.
$
Since \((Z^a)'\) is bounded, it follows that
\[
|
( Z^\tau_{s,\bigcdot})'_r
|
\le
C| Z^{\eta'}_{s,r}|
\le
C|r-s|^{|\eta'|\alpha}
=
C|r-s|^{(|\tau|-1)\alpha}.
\]

We next estimate the local H\"older seminorm of the Gubinelli
derivative. For \(s\le u<v\le t\), the controlled expansion and the
induction hypothesis give
\[
|
( Z^{\eta'}_{s,\bigcdot})_{u,v}
|
\le
|
( Z^{\eta'}_{s,\bigcdot})'_u
|
|X_{u,v}|
+
|
R^{ Z^{\eta'}_{s,\bigcdot}}_{u,v}
|
\le
C|t-s|^{(|\eta'|-1)\alpha}|v-u|^\alpha.
\]
Using
$
( Z^\tau_{s,\bigcdot})'_r
=
Z^{\eta'}_{s,r}(Z^a)'_r,
$
we obtain
\[
|
( Z^\tau_{s,\bigcdot})'_v
-
( Z^\tau_{s,\bigcdot})'_u
|
\le
|
( Z^{\eta'}_{s,\bigcdot})_{u,v}
|
|(Z^a)'_v|
+
| Z^{\eta'}_{s,u}|
|
(Z^a)'_v-(Z^a)'_u
|
\le
C|t-s|^{(|\eta'|-1)\alpha}|v-u|^\alpha.
\]
Therefore,
\[
\|
(Z^\tau_{s,\bigcdot})'
\|_{\alpha;[s,t]}
\le
C|t-s|^{(|\eta'|-1)\alpha}
=
C|t-s|^{(|\tau|-2)_+\alpha}.
\]
It remains to estimate the controlled remainder. By the local
controlled rough integral expansion,
\[
( Z^\tau_{s,\bigcdot})_{u,v}
=
 Z^{\eta'}_{s,u}Z^a_{u,v}
+
(
(Z^{\eta'}_{s,\bigcdot})'_u
\diamond
(Z^a)'_u
)\mathbb X_{u,v}
+
\mathcal E_{u,v},
\]
where
$$
|\mathcal E_{u,v}|
\le
C|t-s|^{(|\eta'|-2)_+\alpha}|v-u|^{3\alpha}.
$$
Since
$$
Z^a_{u,v}
=
(Z^a)'_uX_{u,v}
+
R^{Z^a}_{u,v}, \qquad 
(Z^\tau_{s,\bigcdot})'_u
=
Z^{\eta'}_{s,u}(Z^a)'_u,
$$
it follows that
\[
R^{Z^\tau_{s,\bigcdot}}_{u,v}
=
Z^{\eta'}_{s,u}R^{Z^a}_{u,v}+
(
( Z^{\eta'}_{s,\bigcdot})'_u
\diamond
(Z^a)'_u
)\mathbb X_{u,v}
+
\mathcal E_{u,v}.
\]
The induction hypothesis and the bounds for \(Z^a\) and \(\mathbb X\)
give
\[
|
Z^{\eta'}_{s,u}R^{Z^a}_{u,v}
|
\le
C|t-s|^{|\eta'|\alpha}|v-u|^{2\alpha},
\qquad
|
(
(Z^{\eta'}_{s,\bigcdot})'_u
\diamond
(Z^a)'_u
)\mathbb X_{u,v}
|
\le
C|t-s|^{(|\eta'|-1)\alpha}|v-u|^{2\alpha}.
\]
Moreover, since $|v-u|\le |t-s|$, we have 
$$  
|\mathcal E_{u,v}|
\le
C|t-s|^{(|\eta'|-1)\alpha}|v-u|^{2\alpha}.
$$
Consequently,
$$
\|
R^{ Z^\tau_{s,\bigcdot}}
\|_{2\alpha;[s,t]}
\le
C|t-s|^{(|\eta'|-1)\alpha}
=
C|t-s|^{(|\tau|-2)_+\alpha}.
$$
Combining the last two estimates yields
\[
\|
( Z^\tau_{s,\bigcdot})'
\|_{\alpha;[s,t]}
+
\|
R^{ Z^\tau_{s,\bigcdot}}
\|_{2\alpha;[s,t]}
\le
C|t-s|^{(|\tau|-2)_+\alpha}.
\]
Hence all the stated estimates hold for trees of order \(k+1\).
Together with the product-forest case, they hold for every nonempty
forest of order \(k+1\). The induction is complete.

For fixed \(p\), there are only finitely many decorated forests in
\(\mathcal F_{\le p}\). The constant may therefore be chosen uniformly over
all \(\eta\in\mathcal F_{\le p}\). This proves Proposition~\ref{prop:well-posedness-controlled-coefficients}. 
\end{proof}

In analogy to the classical B-series
\cite{Butcher1963,EW1996},
we combine the elementary differentials from
Definition~\ref{def:elementary-differentials}
with the controlled coefficients from
Definition~\ref{def:controlled-forest-coefficients}.
This yields a B-series expansion in which the classical B-series coefficients are replaced by the recursively defined coefficients \(Z^\tau_{s,t}\).

Throughout this section, $p$ denotes the truncation or expansion order according to the context. Additional assumptions on $p$ are imposed when required for the corresponding estimates.

\begin{definition}[Controlled B-series]
\label{def:controlled-b-series}
Let \(\phi:\mathcal T^0\to\mathbb R\) be a weight map satisfying
\(\phi(\etree)=1\), and denote by
$
\mathcal Z=( Z^\eta)_{\eta\in\mathcal F^0}
$
the family of controlled coefficients defined in
Definition~\ref{def:controlled-forest-coefficients}.
For \(y\in\mathbb R^e\) and \(0\le s\le t\le T\), the
\emph{controlled B-series associated with the weight map \(\phi\)
and the family of controlled coefficients \(\mathcal Z\)} is the formal series
\begin{equation*}
B(\phi;y, \mathcal Z_{s,t})
:=
y+
\sum_{\tau\in\mathcal T}
\frac{\phi(\tau)}{\sigma(\tau)}
F(\tau)(y)\,
 Z^\tau_{s,t},
\end{equation*}
where \(F(\tau)\) denotes the elementary differential associated with
\(\tau\), and \(\sigma(\tau)\) is the symmetry factor of $\tau$.
\end{definition}

For later use, we also define the \emph{truncated controlled B-series of order \(p\)} by
$$
B_p(\phi;y,\mathcal Z_{s,t})
:=
y
+
\sum_{\tau\in\mathcal T_{\leq p}}
\frac{\phi(\tau)}{\sigma(\tau)}
F(\tau)(y)\,
 Z^\tau_{s,t}.
$$
\begin{remark}
If \(Z=X\) and \(Z'=\operatorname{Id}\), then the controlled coefficient
\(Z^\tau_{s,t}\), where
\(Z^\tau_{s,\bigcdot}:[s,T]\to\mathbb R\),
\(r\mapsto Z^\tau_{s,r}\),
reduces to the classical tree coefficient of the rough path expansion.
Moreover, \(F(\tau)\) becomes the usual elementary differential in
Runge-Kutta theory. Hence the controlled B-series reduces to the
classical B-series in this case.
\end{remark}

The notation above separates the two ingredients of the expansion: the elementary
differentials \(F(\tau)\), which depend only on the vector fields \(F_1,\ldots,F_m\), and the
coefficients \(Z^\tau_{s,t}\), which are generated by the controlled driver \(Z\).
This notation will be used to express the local expansion of the exact solution
and to compare it with the Runge-Kutta one-step maps introduced below. It is the
controlled-driver analogue of the B-series formalism used for classical rough
differential equations driven by branched rough paths; see
\cite{RR2022}.

\begin{theorem} (Exact controlled B-series expansion of solutions)
\label{thm:exact-controlled-bseries}
Let \((Y,Y')\in\mathcal D_\bX^\alpha([0,T];\mathbb R^e)\) be the solution of
\eqref{eq:component-controlled-rde} and 
$
F_a\in C_b^{p+1}(\mathbb R^e;\mathbb R^e) 
$
for  $a\in\mathcal A$, where $
(p+1)\alpha>1
$. Then, for every \(0\le s\le t\le T\),
\begin{equation}
\label{eq:exact-controlled-bseries-expansion}
Y_t
=
Y_s
+
\sum_{\tau\in\mathcal T_{\leq p}}
\frac{1}{\sigma(\tau)}
F(\tau)(Y_s)\,
Z^\tau_{s,t}
+
R_{s,t}^{(p+1)}=B_p(1;Y_s,\mathcal Z_{s,t})
+
R_{s,t}^{(p+1)},
\end{equation}
where $1$ denotes the constant weight map, \(F(\tau)\) is the
elementary differential defined in Definition~\ref{def:elementary-differentials}, and the remainder $R_{s,t}^{(p+1)}$ satisfies
$$
|R_{s,t}^{(p+1)}|
\le
C_{p,T}|t-s|^{(p+1)\alpha}.
$$
\end{theorem}

\begin{proof}
Fix \(0\le s\le t\le T\). By \meqref{eq:controlled-driven-fixed-point}, we see that
\begin{equation}
\label{eq:solution-integral-start}
Y_t-Y_s
=
\sum_{a=1}^m
\sint_s^t F_a(Y_r)\,d\bZ_r^a.
\end{equation}
For \(q=0,\ldots,p\), let \(\mathcal P_q\) denote the assertion that, uniformly for
\(r\in[s,t]\),
\begin{equation}
\label{eq:induction-statement}
Y_r
=
Y_s
+
\sum_{\tau\in\mathcal T_{\leq q}}
\frac{1}{\sigma(\tau)}
F(\tau)(Y_s)\,
 Z^\tau_{s,r}
+
O(|r-s|^{(q+1)\alpha}).
\end{equation}
Here and afterward, the remaining estimates are uniform for \(0\le s\le r\le t\le T\).
For \(q=0\), the sum over \(\mathcal T_{\le 0}\) is empty. Since \(Y\) is \(\alpha\)-H\"older,
$
|Y_r-Y_s|\lesssim |r-s|^\alpha.
$
Thus \(\mathcal P_0\) holds.
Assume that \(\mathcal P_q\) is true for some \(0\le q\le p-1\). We prove
\(\mathcal P_{q+1}\). 

Fix \(a\in\mathcal A\). Taylor's formula at \(Y_s\) gives 
\[
F_a(Y_r)
=
\sum_{n=0}^{q}
\frac{1}{n!}
D^nF_a(Y_s)
\bigl[
Y_{s,r},\ldots,Y_{s,r}
\bigr]
+
O(|Y_{s,r}|^{q+1}).
\]
Since \(Y\) is \(\alpha\)-H\"older continuous,
$
|Y_{s,r}|^{q+1}
\lesssim
|r-s|^{(q+1)\alpha}.
$
Using the induction hypothesis \eqref{eq:induction-statement}, we substitute the
tree expansion of \(Y_{s,r}\) into the Taylor polynomial. The products of the tree
terms are encoded by forests. Keeping all products in total tree order at most
\(q\), we obtain
\begin{equation}
\label{eq:integrand-forest-expansion}
F_a(Y_r)
=
\sum_{\substack{\eta\in\mathcal F^0\\ |\eta|\le q}}
\frac{1}{\sigma([\eta]_a)}
F([\eta]_a)(Y_s)\,
 Z^\eta_{s,r}
+
O(|r-s|^{(q+1)\alpha}).
\end{equation}
The coefficient
\(\sigma([\eta]_a)^{-1}\)
in \eqref{eq:integrand-forest-expansion}
follows from the same combinatorial argument as in the definition of the
symmetry factor. Indeed, if
\(\eta=\rho_1^{r_1}\cdots\rho_\ell^{r_\ell}\),
then the multinomial coefficient from the Taylor expansion cancels the
factor \(n!\) and gives
\[
\frac{1}{r_1!\cdots r_\ell!}
\prod_{j=1}^{\ell}
\frac{1}{\sigma(\rho_j)^{r_j}}
=
\frac{1}{\sigma([\eta]_a)}.
\]

We now justify the integration of the remainder term in \eqref{eq:integrand-forest-expansion}. For this, define
\[
E_{a,q}^s(r)
:=
F_a(Y_r)
-
\sum_{\substack{\eta\in\mathcal F^0\\ |\eta|\le q}}
\frac{1}{\sigma([\eta]_a)}
F([\eta]_a)(Y_s)\,
Z^\eta_{s,r},
\quad r\in[s,t].
\]
Since \( Z^\etree_{s,s}=1\) and \( Z^\eta_{s,s}=0\) for every
non-empty forest \(\eta\), it follows that
$
E_{a,q}^s(s)=0.
$
Then \eqref{eq:integrand-forest-expansion} can be written as
\[
F_a(Y_r)
=
\sum_{\substack{\eta\in\mathcal F^0\\ |\eta|\le q}}
\frac{1}{\sigma([\eta]_a)}
F([\eta]_a)(Y_s)\,
Z^\eta_{s,r}
+
E_{a,q}^s(r).
\]

By Lemma~\ref{lem:composition} and
Proposition~\ref{prop:well-posedness-controlled-coefficients}, 
\(F_a(Y)\) and $
\mathbf Z^\eta_{s,\bigcdot}
=
\left(
Z^\eta_{s,\bigcdot},
(Z^\eta_{s,\bigcdot})'
\right)
$ 
for every
\(\eta\in\mathcal F^0\) are \(\bX\)-controlled. Hence
\((E_{a,q}^s,(E_{a,q}^s)')\) is also an \(\bX\)-controlled rough path on
\([s,t]\).
For \(q=0\), the required estimates follow directly from the composition
estimate applied to
\(E_{a,0}^s(r)=F_a(Y_r)-F_a(Y_s)\). In the remainder of the
argument, let \(1\le q\le p-1\).
The Gubinelli derivative of \(E_{a,q}^s\) is
$$
(E_{a,q}^s)'_r
=
DF_a(Y_r)Y'_r-
\sum_{\substack{\eta\in\mathcal F^0\\ |\eta|\le q}}
\frac{1}{\sigma([\eta]_a)}
F([\eta]_a)(Y_s)
(Z^\eta_{s,\bigcdot})'_r.
$$
Since
$
Y'_r=\sum_{b=1}^{m}F_b(Y_r)(Z^b)'_r,
$
the first term becomes
$$
DF_a(Y_r)Y'_r
=
\sum_{b=1}^{m}
DF_a(Y_r)F_b(Y_r)(Z^b)'_r.
$$
For the empty forest, the corresponding derivative is zero.
For each \(b=1,\ldots,m\), expand
\(DF_a(Y_r)F_b(Y_r)\) at \(Y_s\) up to order \(q-1\), and substitute
the induction expansion of \(Y_{s,r}\). By the same calculation as in~\eqref{eq:integrand-forest-expansion},
the forest terms of order at most \(q-1\) cancel. Therefore, only terms
of order at least \(q\) and the Taylor remainder
\(O(|Y_{s,r}|^q)\) remain. Since
$
|Y_{s,r}|^q\le C|r-s|^{q\alpha},
$
Proposition~\ref{prop:well-posedness-controlled-coefficients} and the
boundedness of \((Z^b)'\) yield
\[
|(E_{a,q}^s)'_r|
\le
C|r-s|^{q\alpha}.
\]
Together with~\eqref{eq:integrand-forest-expansion}, the definition of
\(E_{a,q}^s\) gives
$
|E_{a,q}^s(r)|
\le
C|r-s|^{(q+1)\alpha}.
$

Let \(s\le u<v\le t\). Applying the above expansion at \(u\) and \(v\)
and using Proposition~\ref{prop:well-posedness-controlled-coefficients},
we obtain
$$
|(E_{a,q}^s)'_v-(E_{a,q}^s)'_u|
\le
C|t-s|^{(q-1)\alpha}|v-u|^\alpha.
$$
Hence
$$
\|(E_{a,q}^s)'\|_{\alpha;[s,t]}
\le
C|t-s|^{(q-1)\alpha}.
$$
Similarly, using the controlled expansions of \(F_a(Y)\) and
\(\mathbf Z^\eta_{s,\bigcdot}\), together with the same cancellation as
above, we obtain
$$
|R^{E_{a,q}^s}_{u,v}|
\le
C|t-s|^{(q-1)\alpha}|v-u|^{2\alpha}.
$$
Therefore,
\[
\|R^{E_{a,q}^s}\|_{2\alpha;[s,t]}
\le
C|t-s|^{(q-1)\alpha}.
\]
Combining these estimates with the case \(q=0\) gives
\[
\|(E_{a,q}^s)'\|_{\alpha;[s,t]}
+
\|R^{E_{a,q}^s}\|_{2\alpha;[s,t]}
\le
C|t-s|^{(q-1)_+\alpha}.
\]
Here the constant \(C\) is uniform for \(0\le s\le r\le t\le T\). Since
\(q\le p-1\), the required derivatives are bounded by the assumption
\(F_a\in C_b^{p+1}(\mathbb R^e;\mathbb R^e)\).
Applying Lemma~\ref{lem:controlled-remainder-integration} with
$
\beta=(q+1)\alpha
$
yields
\[
\left|
\sint_s^t E_{a,q}^s(r)\,d\bZ_r^a
\right|
\le
C|t-s|^{(q+2)\alpha}.
\]
Consequently, by Definition~\ref{def:controlled-forest-coefficients},  
\[
\begin{aligned}
\sint_s^t F_a(Y_r)\,d\bZ_r^a
&=
\sum_{\substack{\eta\in\mathcal F^0\\ |\eta|\le q}}
\frac{1}{\sigma([\eta]_a)}
F([\eta]_a)(Y_s)
\sint_s^t  Z^\eta_{s,r}\,d\bZ_r^a
+
O(|t-s|^{(q+2)\alpha})
\\
&=
\sum_{\substack{\eta\in\mathcal F^0\\ |\eta|\le q}}
\frac{1}{\sigma([\eta]_a)}
F([\eta]_a)(Y_s)
 Z^{[\eta]_a}_{s,t}
+
O(|t-s|^{(q+2)\alpha}).
\end{aligned}
\]
Summing over \(a=1,\ldots,m\) in \eqref{eq:solution-integral-start}, we get
\[
Y_t-Y_s
=
\sum_{a=1}^m
\sum_{\substack{\eta\in\mathcal F^0\\ |\eta|\le q}}
\frac{1}{\sigma([\eta]_a)}
F([\eta]_a)(Y_s)
 Z^{[\eta]_a}_{s,t}
+
O(|t-s|^{(q+2)\alpha}).
\]
Every tree \(\tau\in\mathcal T_{\le q+1}\) can be written uniquely as
$
\tau=[\eta]_a
$
for some \(a\in\mathcal A\) and some \(\eta\in\mathcal F^0\) satisfying
$
|\eta|\le q.
$
Hence
\[
Y_t
=
Y_s
+
\sum_{\tau\in\mathcal T_{\le q+1}}
\frac{1}{\sigma(\tau)}
F(\tau)(Y_s)
Z^\tau_{s,t}
+
O(|t-s|^{(q+2)\alpha}).
\]
This proves \(\mathcal P_{ q+1}\).
By induction, \(\mathcal P_p\) holds. Therefore,
\[
Y_t
=
Y_s
+
\sum_{\tau\in\mathcal T_{\le p}}
\frac{1}{\sigma(\tau)}
F(\tau)(Y_s)
 Z^\tau_{s,t}
+
R_{s,t}^{(p+1)},
\]
where
$$
|R_{s,t}^{(p+1)}|
\le
C_{p,T}|t-s|^{(p+1)\alpha}.
$$
The equivalent formulation
$
Y_t
=
B_p(1;Y_s,\mathcal Z_{s,t})
+
R_{s,t}^{(p+1)}
$
follows directly from Definition~\ref{def:controlled-b-series}. This completes the proof of Theorem~\ref{thm:exact-controlled-bseries}.
\end{proof}

We next introduce an abstract controlled B-series form for a generic one-step
approximation. This provides the mechanism for deriving local order conditions
by comparison with the exact expansion in
Theorem~\ref{thm:exact-controlled-bseries}.

\begin{definition}
\label{def:numerical-controlled-bseries}
Let $p\in \NN$ and
$
\Psi=\{\Psi_{s,t}\}_{0\le s\le t\le T}
$ be a family of one-step maps 
$
\Psi_{s,t}:\mathbb R^e\to\mathbb R^e
$. We say that \(\Psi\) \emph{admits a controlled B-series
expansion of order \(p\)} if there exist a weight map
\(\phi_\Psi:\mathcal T_{\leq p}\to\mathbb R\) and remainder maps
\(R_{\Psi,s,t}^{(p+1)}:\mathbb R^e\to\mathbb R^e\) such that, for all
\(y\in\mathbb R^e\) and \(0\le s\le t\le T\),
\begin{equation}\label{eq:numerical-controlled-bseries}
\Psi_{s,t}(y)
=
B_p(\phi_\Psi;y,\mathcal Z_{s,t})
+
R_{\Psi,s,t}^{(p+1)}(y)
=
y+
\sum_{\tau\in\mathcal T_{\le p}}
\frac{\phi_\Psi(\tau)}{\sigma(\tau)}
F(\tau)(y) Z^\tau_{s,t}
+
R_{\Psi,s,t}^{(p+1)}(y).
\end{equation}
The remainder is assumed to satisfy
\begin{equation*}
|R_{\Psi,s,t}^{(p+1)}(y)|
\le
C_{p,T}(1+|y|)|t-s|^{(p+1)\alpha},
\end{equation*}
where \(C_{p,T}>0\) is independent of \(s,t\) and \(y\).
\end{definition}

\begin{proposition}
\label{prop:tree-matching}
Let $(Y, Y')$ be the exact solution of \eqref{eq:component-controlled-rde} and
\(\Psi\)  admit a controlled B-series
expansion of order \(p\). If
$ \phi_\Psi(\tau)=1$ for $\tau\in\mathcal T_{\le p},
$
then the local truncation error satisfies
\begin{equation}
\label{eq:local-truncation-error}
|Y_t-\Psi_{s,t}(Y_s)|
\le
C_{p,T}(1+|Y_s|)\,|t-s|^{(p+1)\alpha},
\qquad 0\le s\le t\le T.
\end{equation}
\end{proposition}

\begin{proof}
By Theorem~\ref{thm:exact-controlled-bseries}, we see that
$$
Y_t
=
Y_s
+
\sum_{\tau\in\mathcal T_{\le p}}
\frac{1}{\sigma(\tau)}
F(\tau)(Y_s)\, Z^\tau_{s,t}
+
R_{s,t}^{(p+1)},
$$
where
$
|R_{s,t}^{(p+1)}|
\le
C_{p,T}|t-s|^{(p+1)\alpha}.
$
On the other hand, from \eqref{eq:numerical-controlled-bseries} with
\(y=Y_s\), it follows that
\[
\Psi_{s,t}(Y_s)
=
Y_s
+
\sum_{\tau\in\mathcal T_{\le p}}
\frac{\phi_\Psi(\tau)}{\sigma(\tau)}
F(\tau)(Y_s)\, Z^\tau_{s,t}
+
R_{\Psi,s,t}^{(p+1)}(Y_s).
\]
Subtracting the two identities gives
\[
Y_t-\Psi_{s,t}(Y_s)
=
\sum_{\tau\in\mathcal T_{\le p}}
\frac{1-\phi_\Psi(\tau)}{\sigma(\tau)}
F(\tau)(Y_s)\, Z^\tau_{s,t}
+
R_{s,t}^{(p+1)}
-
R_{\Psi,s,t}^{(p+1)}(Y_s).
\]
Since \(\phi_\Psi(\tau)=1\) for every \(\tau\in\mathcal T_{\le p}\), the sum vanishes term by term. Hence,
\[
|Y_t-\Psi_{s,t}(Y_s)|
\le
|R_{s,t}^{(p+1)}|
+
|R_{\Psi,s,t}^{(p+1)}(Y_s)|
\le
C_{p,T}(1+|Y_s|)\,|t-s|^{(p+1)\alpha},
\]
which completes the proof of \eqref{eq:local-truncation-error}.
\end{proof}

\section{Controlled Runge-Kutta methods}
\label{sec:runge-kutta-methods}
In this section, we introduce controlled Runge-Kutta methods whose coefficients may depend on the controlled coefficients $\{Z_{u,v}^{\tau}:\tau\in\mathcal T_{
\le p}\}. $ The exact controlled B-series expansion in Theorem~\ref{thm:exact-controlled-bseries} identifies the tree coefficients that a one-step approximation must reproduce. By comparing the B-series expansion of the numerical method with this exact expansion, we derive the corresponding order conditions and local error estimates.

Let \([u,v]\subset[0,T]\) be a one-step interval. We use a generic interval
because the local truncation error analysis is performed locally.

\begin{definition} (Controlled Runge-Kutta methods)
\label{def:controlled-rk}
Fix \(q\in\mathbb N\). For each
\([u,v]\subset[0,T]\) and \(a\in\mathcal A\), let
\(
\mathcal Z_{u,v}^{(a)}\in\mathbb R^{q\times q}\) and 
\(\zeta_{u,v}^{(a)}\in\mathbb R^q,
\)
whose entries may depend on the stepwise coefficients
\(
\{ Z^\eta_{u,v}:\eta\in\mathcal F_{\leq p}^0\}.
\)
For \(y\in\mathbb R^e\), the \emph{stage values} \(Y_i^{u,v}(y)\in\mathbb R^e\) are defined by  
\begin{equation}
\label{eq:controlled-rk-stages}
Y_i^{u,v}(y)
:=
y+
\sum_{a=1}^m
\sum_{j=1}^q
\zuvij
F_a\bigl(Y_j^{u,v}(y)\bigr),
\qquad
i=1,\ldots,q.
\end{equation}
The associated \emph{controlled Runge-Kutta map} $\Psi_{u,v}^{\crk}:\RR^e\rightarrow \RR^e$ is given by  
\begin{equation}
\label{eq:controlled-rk-one-step}
\Psi_{u,v}^{\crk}(y)
:=
y+
\sum_{a=1}^m
\sum_{i=1}^q
\zeuvi
F_a\bigl(Y_i^{u,v}(y)\bigr).
\end{equation}
\end{definition}

\begin{remark}
If the driving signal is reduced to the time variable,
i.e. \(m=1\) and \(Z_t=t\), and the controlled Runge-Kutta
coefficients are chosen as
\[
(\mathcal Z^{(a)}_{u,v})_{i,j}=(v-u)a_{ij},
\qquad
(\zeta^{(a)}_{u,v})_i=(v-u)b_i,
\]
where \(a_{ij}\) and \(b_i\) are the classical Runge-Kutta
coefficients, then the controlled Runge-Kutta scheme reduces to
the classical Runge-Kutta method with step size \(v-u\).
\end{remark}
 
\begin{assumption}
\label{ass:rk-solvability-size}
Let \(|\cdot|_\infty\) be the max norm on \(\mathbb R^q\) and 
\(\|\cdot\|_\infty\) the matrix norm on \(\mathbb R^{q\times q}\).
We assume that there exists a constant \(C_{\mathrm{CRK}}>0\) such that, for every
\([u,v]\subset[0,T]\) and every 
\(a\in\mathcal A\),
\begin{equation}
\label{eq:rk-coefficient-size}
\|\mathcal Z_{u,v}^{(a)}\|_\infty
+
|\zeta_{u,v}^{(a)}|_\infty
\le
C_{\mathrm{CRK}}|v-u|^\alpha.
\end{equation}
We also assume that the stage system
\eqref{eq:controlled-rk-stages} admits a solution on each one-step interval under consideration.
\end{assumption}
This assumption will be verified for the specific coefficients considered
below. To compare the numerical one-step map with the exact controlled
B-series, we introduce the associated coefficients following the
algebraic viewpoint of Runge-Kutta methods \cite{Butcher1972}. These
coefficients are the numerical counterparts of the coefficients
\(Z^\tau_{u,v}\) in the exact expansion and may depend on controlled
coefficients in the present setting.
We use \(\odot\) to denote the componentwise product on \(\mathbb R^q\):
\[
(x\odot y)_i :=x_i y_i,\qquad x,y\in\mathbb R^q, \qquad  i=1,\ldots,q,
\]
with the convention that the empty product is
\(\mathbf 1_q=(1,\ldots,1)^\top \in \RR^q\).

\begin{definition}\label{def:rk-tree-coefficients}
For \(a\in\mathcal A\), let
\(
\mathcal Z_{u,v}^{(a)}\in\mathbb R^{q\times q}\) and 
\(\zeta_{u,v}^{(a)}\in\mathbb R^q.
\)
For \(0\le u\le v\le T\), the \emph{recursive coefficient maps}
\(\Phi_{u,v}:\mathcal T^0\to\mathbb R^q\) and
\(a_{u,v}:\mathcal T^0\to\mathbb R\)
are defined as follows.
\begin{enumerate}
\item[\rm(i)]
For the empty tree \(\etree\), let
\[
\Phi_{u,v}(\etree) :=\mathbf 1_q,
\qquad
a_{u,v}(\etree) :=1 .
\]

\item[\rm(ii)]
For a tree
\(
\tau=[\tau_1\cdots\tau_n]_a,
\)
set
\begin{equation}
\label{eq:stage-tree-coefficient}
\Phi_{u,v}(\tau)
:=
\mathcal Z_{u,v}^{(a)}
\bigl(
\Phi_{u,v}(\tau_1)\odot\cdots\odot
\Phi_{u,v}(\tau_n)
\bigr),
\qquad
a_{u,v}(\tau)
:=
\left\langle
\zeta_{u,v}^{(a)},
\Phi_{u,v}(\tau_1)\odot\cdots\odot
\Phi_{u,v}(\tau_n)
\right\rangle_{\mathbb R^q}.
\end{equation}
\end{enumerate}
\end{definition}

In particular, for the single-node tree \(\bullet_a=[\etree]_a\),  
\[
\Phi_{u,v}(\bullet_a)
=
\mathcal Z_{u,v}^{(a)}\mathbf 1_q,
\qquad
a_{u,v}(\bullet_a)
=
\left\langle
\zeta_{u,v}^{(a)},\mathbf 1_q
\right\rangle_{\mathbb R^q}.
\]

\begin{proposition}
\label{prop:rk-tree-coefficient-scaling}
For every \(p\in\mathbb Z_{\ge1}\),
there exist constants $C_{p,\mathrm{CRK}}$  and $\widetilde{C}_{p,\mathrm{CRK}}>0$ such that, for all
\(\tau\in\mathcal T_{\le p}^0\) and all \(0\le u\le v\le T\),
\begin{equation}
\label{eq:stage-coefficient-scaling}
|\Phi_{u,v}(\tau)|_\infty
\le
C_{p,\mathrm{CRK}}|v-u|^{|\tau|\alpha},
\qquad
|a_{u,v}(\tau)|
\le
\widetilde{C}_{p,\mathrm{CRK}}|v-u|^{|\tau|\alpha}.
\end{equation}
\end{proposition}

\begin{proof}
We use induction on the order \(|\tau|\geq0\).
The estimates are first established for each fixed tree \(\tau\).
Since \(\mathcal T_{\le p}^0\) is finite for fixed \(p\), the corresponding
constants can be chosen uniformly over all
\(\tau\in\mathcal T_{\le p}^0\).

For the empty tree \(\tau=\etree\), we have \(|\etree|=0\) by
\eqref{C3-1}. Definition~\ref{def:rk-tree-coefficients} gives
$
\Phi_{u,v}(\etree)=\mathbf 1_q
$ and $
a_{u,v}(\etree)=1.
$
Since \(|\mathbf 1_q|_\infty=1\), both estimates hold for
\(\tau=\etree\).
Now let $\tau=[\tau_1\cdots\tau_n]_a$ with $n\geq 0$ and $a\in \mathcal{A}$. Then 
$$
|\tau|=1+\sum_{j=1}^n|\tau_j|.
$$
Suppose that the estimates hold for all trees of order strictly
smaller than \(|\tau|\). By
\eqref{eq:stage-tree-coefficient}, the compatibility of the induced
matrix norm with the max norm, and
$$
\left|
x^{(1)}\odot\cdots\odot x^{(n)}
\right|_\infty
\leq
\prod_{j=1}^n |x^{(j)}|_\infty,
$$
we obtain
$$
|\Phi_{u,v}(\tau)|_\infty
\leq
\|\mathcal Z_{u,v}^{(a)}\|_\infty
\prod_{j=1}^n
|\Phi_{u,v}(\tau_j)|_\infty
\leq
C_{\mathrm{CRK}}|v-u|^\alpha
\prod_{j=1}^n
C_{\tau_j}|v-u|^{|\tau_j|\alpha}
=
C_\tau |v-u|^{|\tau|\alpha},
$$
where
$
C_\tau
:=
C_{\mathrm{CRK}}\prod_{j=1}^n C_{\tau_j}.
$
Similarly, by \eqref{eq:stage-tree-coefficient} and
$
\left|
\langle\zeta,x\rangle_{\mathbb R^q}
\right|
\leq
q|\zeta|_\infty|x|_\infty,
$
we have
$$
|a_{u,v}(\tau)|
\leq
q|\zeta_{u,v}^{(a)}|_\infty
\prod_{j=1}^n
|\Phi_{u,v}(\tau_j)|_\infty
\leq
qC_{\mathrm{CRK}}|v-u|^\alpha
\prod_{j=1}^n
C_{\tau_j}|v-u|^{|\tau_j|\alpha}
=
\widetilde C_\tau |v-u|^{|\tau|\alpha},
$$
where
$
\widetilde C_\tau
:=
qC_{\mathrm{CRK}}\prod_{j=1}^n C_{\tau_j}.
$
This completes the induction.
Since \(\mathcal T_{\le p}^0\) is finite for each fixed \(p\), we may set
\[
C_{p,\mathrm{CRK}}
:=
\max_{\tau\in\mathcal T_{\le p}^0}C_\tau,
\qquad
\widetilde C_{p,\mathrm{CRK}}
:=
\max_{\tau\in\mathcal T_{\le p}^0}\widetilde C_\tau.
\]
The estimates
\eqref{eq:stage-coefficient-scaling} then hold uniformly for all
\(\tau\in\mathcal T_{\le p}^0\).
\end{proof}

We now derive the controlled B-series expansion of the
numerical map \(\Psi_{u,v}^{\crk}\) in~\eqref{eq:controlled-rk-one-step}. The argument is the usual rooted-tree Taylor
expansion for Runge-Kutta methods, with the stepwise matrices and weights
\(\mathcal Z_{u,v}^{(a)}\) and \(\zeta_{u,v}^{(a)}\) replacing the classical
Runge-Kutta coefficients.

\begin{theorem} (Numerical controlled B-series expansion of the controlled Runge-Kutta map) \label{thm:rk-bseries} 
For $(p+1) \alpha>1$, let $
F_a\in C_b^{p+1}(\mathbb R^e;\mathbb R^e) 
$ with $a\in\mathcal A$. Suppose that Assumption~\ref{ass:rk-solvability-size} holds. 
Then, for any \(R>0\), there exists a constant \(C_{p,R,T}>0\) such that, for any
\(0\le u\le v\le T\), and any initial state
\(y\in\mathbb R^e\) for the Runge-Kutta step on \([u,v]\)
satisfying \(|y|\le R\), the stage values in~\eqref{eq:controlled-rk-stages} satisfy
\begin{equation}
\label{eq:stage-bseries-expansion}
Y_i^{u,v}(y)
=
y+
\sum_{\tau\in\mathcal T_{\leq p}}
\frac{\bigl(\Phi_{u,v}(\tau)\bigr)_i}{\sigma(\tau)}
F(\tau)(y)
+
\mathcal R_{i,u,v}^{(p+1)}(y), \qquad i=1,\ldots, q,
\end{equation}
where \(F(\tau)\) is the elementary differential associated with the
vector fields \(F_a\) defined in Definition~\ref{def:elementary-differentials}, and the controlled Runge-Kutta map $\Psi_{u,v}^{\crk}$ in~\eqref{eq:controlled-rk-one-step} satisfies
\begin{equation}
\label{eq:rk-bseries-expansion}
\Psi_{u,v}^{\crk}(y)
=
y+
\sum_{\tau\in\mathcal T_{\leq p}}
\frac{a_{u,v}(\tau)}{\sigma(\tau)}
F(\tau)(y)
+
\mathcal R_{u,v}^{(p+1)}(y). 
\end{equation}
Moreover,  
\begin{equation}
\label{eq:rk-bseries-remainder}
\max_{1\le i\le q}
|
\mathcal R_{i,u,v}^{(p+1)}(y)
|
+
|
\mathcal R_{u,v}^{(p+1)}(y)
|
\le
C_{p,R,T}|v-u|^{(p+1)\alpha}.
\end{equation}
\end{theorem}

\begin{proof} 
Fix $0 \leq u \leq v \leq T$. In the interval $[u, v]$, the matrices $\mathcal{Z}_{u, v}^{(a)}$ and the vectors $\zeta_{u, v}^{(a)}$ with $a\in \mathcal{A}$ are fixed numerical coefficients. Therefore, the classical rooted-tree Taylor expansion for Runge-Kutta methods applies to the stage equations; see, for example, \cite{HLW2006}.

By \eqref{eq:controlled-rk-stages}, we deduce that
$$
Y_i^{u, v}(y)-y=\sum_{a=1}^m \sum_{j=1}^q (\mathcal{Z}^{(a)}_{u, v})_{i,j}F_a(Y_j^{u, v}(y)) .
$$
Since $F_a \in C_b^{p+1}\left(\mathbb{R}^e ; \mathbb{R}^e\right)$, each $F_a$ is bounded. Assumption~\ref{ass:rk-solvability-size} gives
$
\|\mathcal{Z}_{u, v}^{(a)}\|_{\infty} \leq C_{\mathrm{CRK}}|v-u|^\alpha.
$
Consequently,
\begin{equation}\label{E4-2}
|Y_i^{u, v}(y)-y|  \leq \sum_{a=1}^m\|\mathcal{Z}_{u, v}^{(a)}\|_{\infty} \max _{1 \leq j \leq q}|F_a(Y_j^{u, v}(y))| 
 \leq \sum_{a=1}^m\|\mathcal{Z}_{u, v}^{(a)}\|_{\infty}\|F_a\|_{\infty} 
 \leq C|v-u|^\alpha,
\end{equation}
uniformly for $|y| \leq R$ and $i=1, \ldots, q$. Hence $$\max _{1 \leq i \leq q}|Y_i^{u, v}(y)-y| \leq C|v-u|^\alpha.$$
For all $a=1, \ldots, m$ and $j=1, \ldots, q$, Taylor's formula at $y$ yields
\begin{equation}
\label{E4-1}
F_a(Y_j^{u, v}(y))=  \sum_{n=0}^{p-1} \frac{1}{n!} D^n F_a(y)(Y_j^{u, v}(y)-y, \ldots, Y_j^{u, v}(y)-y) 
 +O(|Y_j^{u, v}(y)-y|^p) .
\end{equation}
Since
$
|Y_j^{u, v}(y)-y| \leq C|v-u|^\alpha,
$ the Taylor remainder satisfies  
$$
O(|Y_j^{u, v}(y)-y|^p)=O(|v-u|^{p \alpha}) .
$$
After multiplication by
$(\mathcal{Z}^{(a)}_{u, v})_{i,j}=O (|v-u|^\alpha )$, the corresponding contribution to \eqref{eq:controlled-rk-stages} is 
$
O (|v-u|^{(p+1) \alpha} ) .
$

We identify the terms in the Taylor expansion by induction on the tree order for $|\tau| \geq 1$.
For the one-node tree
$
\bullet_a=[\mathbf{1}]_a,
$
Definition~\ref{def:rk-tree-coefficients} gives
$
\Phi_{u, v}\left(\bullet_a\right)=\mathcal{Z}_{u, v}^{(a)} \mathbf{1}_q$ and $\mathbf{1}_q=(1, \ldots, 1)^{\top}.
$
By
$
F\left(\bullet_a\right)(y)=F_a(y)$ and
$\sigma\left(\bullet_a\right)=1,$ we have
$$
\sum_{a=1}^m \sum_{j=1}^q (\mathcal{Z}_{u, v}^{(a)})_{i, j} F_a(y) =\sum_{a=1}^m \frac{\left(\Phi_{u, v}\left(\bullet_a\right)\right)_i}{\sigma\left(\bullet_a\right)} F\left(\bullet_a\right)(y) .
$$
Thus, the zeroth-order Taylor term produces exactly the one-node trees.
The remaining part of the Taylor expansion in \eqref{E4-1} contains at least one factor $Y_j^{u, v}(y)-y$. 
The estimate in \eqref{E4-2} shows that
$$
|Y_j^{u, v}(y)-y| \leq C|v-u|^\alpha
.$$ 
After multiplication by $(\mathcal{Z}^{(a)}_{u, v})_{i, j}=O(|v-u|^\alpha)$, the remainder is $O(|v-u|^{2 \alpha})$. Hence,
\begin{equation}\label{E4-3}
Y_i^{u, v}(y)=y+\sum_{a=1}^m \frac{\left(\Phi_{u, v}\left(\bullet_a\right)\right)_i}{\sigma\left(\bullet_a\right)} F\left(\bullet_a\right)(y)+O (|v-u|^{2 \alpha} ) .
\end{equation}

To see how the next tree order is generated, consider the linear Taylor term
$
D F_a(y) (Y_j^{u, v}(y)-y ) .
$
From \eqref{E4-3}, the first-order expansion of the $j$-th stage increment is
$$
Y_j^{u, v}(y)-y=\sum_{b=1}^m \frac{\left(\Phi_{u, v}\left(\bullet_b\right)\right)_j}{\sigma\left(\bullet_b\right)} F\left(\bullet_b\right)(y)+O(|v-u|^{2 \alpha}) .
$$

Assume that for some $r \in\{1, \ldots, p-1\}$, all terms associated with trees of order at most $r$ have been identified, and
\begin{equation}\label{E4-7}
Y_j^{u, v}(y)-y=\sum_{\tau \in \mathcal{T}_{\le r}} \frac{\left(\Phi_{u, v}(\tau)\right)_j}{\sigma(\tau)} F(\tau)(y)+O (|v-u|^{(r+1) \alpha} )
\end{equation}
holds uniformly for $j=1, \ldots, q$ and $|y| \leq R$. We prove that the expansion can be extended to the tree order $r+1$.
Consider the $n$-th Taylor term
$$
\frac{1}{n!} D^n F_a(y)(Y_j^{u, v}(y)-y, \ldots, Y_j^{u, v}(y)-y) . $$
From the $n$ copies of $Y_j^{u, v}(y)-y$, choose the terms associated with the trees
$\tau_1, \ldots, \tau_n$.
According to the induction hypothesis \eqref{E4-7}, the term selected from the $\ell$-th copy is
$$
\frac{\left(\Phi_{u, v}\left(\tau_{\ell}\right)\right)_j}{\sigma\left(\tau_{\ell}\right)} F\left(\tau_{\ell}\right)(y).
$$
Before collecting terms corresponding to the same unordered collection of branches, the contribution to the \(i\)-th stage is
$$
\frac{1}{n!} \sum_{j=1}^q (\mathcal{Z}^{(a)}_{u, v})_{i, j} \prod_{\ell=1}^n \frac{\left(\Phi_{u, v}\left(\tau_{\ell}\right)\right)_j}{\sigma\left(\tau_{\ell}\right)} 
 D^n F_a(y)\left(F\left(\tau_1\right)(y), \ldots, F\left(\tau_n\right)(y)\right).
$$
Let
$
\tau=\left[\tau_1 \cdots \tau_n\right]_a .
$
By Definition~\ref{def:elementary-differentials}, we have
\begin{equation}\label{E4-8}
D^n F_a(y)\left(F\left(\tau_1\right)(y), \ldots, F\left(\tau_n\right)(y)\right) =F\left(\left[\tau_1 \cdots \tau_n\right]_a\right)(y) =F(\tau)(y) .
\end{equation}
The order of $\tau$ is
$
|\tau|=1+\sum_{\ell=1}^n\left|\tau_{\ell}\right|.
$
Hence, branches satisfying
\(\sum_{\ell=1}^n|\tau_\ell|\le r\)
generate trees of order at most \(r+1\),
whereas those satisfying
\(\sum_{\ell=1}^n|\tau_\ell|\ge r+1\)
generate trees of order at least \(r+2\) and are absorbed into the remainder. Next, we identify the numerical coefficient of $\tau$. The $j$-th component of the componentwise product satisfies
$$
\left(\Phi_{u, v}\left(\tau_1\right) \odot \cdots \odot \Phi_{u, v}\left(\tau_n\right)\right)_j=\prod_{\ell=1}^n\left(\Phi_{u, v}\left(\tau_{\ell}\right)\right)_j . 
$$
Therefore,
\begin{equation}\label{E4-9}
\begin{aligned}
\sum_{j=1}^q \mathcal{Z}_{u, v, i j}^{(a)} \prod_{\ell=1}^n\left(\Phi_{u, v}\left(\tau_{\ell}\right)\right)_j 
& =\sum_{j=1}^q \mathcal{Z}_{u, v, i j}^{(a)}\left(\Phi_{u, v}\left(\tau_1\right) \odot \cdots \odot \Phi_{u, v}\left(\tau_n\right)\right)_j  \\
& =\left[\mathcal{Z}_{u, v}^{(a)}\left(\Phi_{u, v}\left(\tau_1\right) \odot \cdots \odot \Phi_{u, v}\left(\tau_n\right)\right)\right]_i \\
& =\left(\Phi_{u, v}(\tau)\right)_i.
\end{aligned}
\end{equation}
It remains to determine the combinatorial factor.

Let \(\rho_1,\ldots,\rho_s\) be the distinct branches of \(\tau\), with
multiplicities \(m_1,\ldots,m_s\), where
\(m_1+\cdots+m_s=n\).
The same unordered collection of branches appears
$
\frac{n!}{m_1!\cdots m_s!}
$
times in the \(n\)-th Taylor term. Combining this multinomial factor with
the Taylor coefficient \(1/n!\) and the symmetry factors from the induction
hypothesis gives
\begin{equation}\label{E4-10}
\frac{1}{m_{1}!\cdots m_{s}!\sigma\left(\rho_1\right)^{m_1} \cdots \sigma\left(\rho_s\right)^{m_s}} =\frac{1}{\sigma(\tau)}. 
\end{equation}
Combining \eqref{E4-8}, \eqref{E4-9} and \eqref{E4-10}, the complete contribution associated with the tree $\tau=\left[\tau_1 \cdots \tau_n\right]_a$ is
$$
\frac{\left(\Phi_{u, v}(\tau)\right)_i}{\sigma(\tau)} F(\tau)(y) .
$$
It remains to estimate the omitted terms. By the induction hypothesis, any term that contains the preceding remainder is
\(O(|v-u|^{(r+1)\alpha})\) before substitution.
Multiplication by \(\mathcal{Z}_{u,v}^{(a)}=O(|v-u|^\alpha)\)
yields \(O(|v-u|^{(r+2)\alpha})\). If $\sum_{\ell=1}^n\left|\tau_{\ell}\right| \geq r+1$, then Proposition~\ref{prop:rk-tree-coefficient-scaling} yields
$$
\|\mathcal{Z}_{u, v}^{(a)}\|_{\infty} \prod_{\ell=1}^n|\Phi_{u, v}(\tau_{\ell})|_{\infty} \leq C|v-u|^{(1+\sum|\tau_{\ell}|) \alpha} \leq C|v-u|^{(r+2) \alpha} .
$$
The Taylor remainder is $O(|v-u|^{(p+1) \alpha})$, which is also $O(|v-u|^{(r+2) \alpha})$.  Hence all omitted terms are $O(|v-u|^{(r+2) \alpha})$.
These estimates show that
$$
Y_i^{u, v}(y)=y+\sum_{\tau \in \mathcal{T}_{\le{r+1}}} \frac{\left(\Phi_{u, v}(\tau)\right)_i}{\sigma(\tau)} F(\tau)(y)+O(|v-u|^{(r+2) \alpha}) .
$$
This completes the induction step.

Proceeding inductively from tree order one to tree order $p$, we obtain
\begin{equation}\label{E4-11}
Y_i^{u, v}(y)=y+\sum_{\tau \in \mathcal{T}_{\le p}} \frac{\left(\Phi_{u, v}(\tau)\right)_i}{\sigma(\tau)} F(\tau)(y)+O(|v-u|^{(p+1) \alpha}) .
\end{equation}
All elementary differentials associated with trees of order at most $p+1$ are uniformly bounded because
$
F_a \in C_b^{p+1}\left(\mathbb{R}^e ; \mathbb{R}^e\right), \, a \in\mathcal A.
$
For fixed $p$, finiteness of decorated trees of order $\leq p+1$ ensures uniform constants on all such trees, $i$, and $|y| \leq R$. Thus, in the theorem's remainder notation, \eqref{E4-11} becomes
$$
Y_i^{u, v}(y)=y+\sum_{\tau \in \mathcal{T}_{\le p}} \frac{\left(\Phi_{u, v}(\tau)\right)_i}{\sigma(\tau)} F(\tau)(y)+\mathcal{R}_{i, u, v}^{(p+1)}(y), 
$$
where
$$
\max _{1 \leq i \leq q}|\mathcal{R}_{i, u, v}^{(p+1)}(y)| \leq C_{p, R, T}|v-u|^{(p+1) \alpha} .
$$
This proves \eqref{eq:stage-bseries-expansion}.

We next consider the one-step map. By
\eqref{eq:controlled-rk-one-step}, we have  
$$
\Psi_{u,v}^{\mathrm{CRK}}(y)-y
=
\sum_{a=1}^{m}\sum_{i=1}^{q}
\zeta_{u,v,i}^{(a)}
F_a(Y_i^{u,v}(y)).
$$
Since
$
|Y_i^{u,v}(y)-y|
\leq C|v-u|^\alpha,
$
it follows that 
$$
F_a(Y_i^{u,v}(y))
=
\sum_{n=0}^{p-1}\frac{1}{n!}
D^nF_a(y)
(
Y_i^{u,v}(y)-y,\ldots,
Y_i^{u,v}(y)-y
)
+
O(
|Y_i^{u,v}(y)-y|^p
).
$$
Hence, the Taylor remainder is
\(O(|v-u|^{p\alpha})\). Since
Assumption~\ref{ass:rk-solvability-size} gives
$
|\zeta_{u,v}^{(a)}|_\infty
\leq
C_{\mathrm{CRK}}|v-u|^\alpha,
$
its contribution to the one-step map is
\(O(|v-u|^{(p+1)\alpha})\).

We now substitute the stage expansion
\eqref{eq:stage-bseries-expansion} into the Taylor
polynomial. Consider a tree
$
\tau=[\tau_1\cdots\tau_n]_a.
$
Choosing from the \(n\) copies of
\(Y_i^{u,v}(y)-y\) the terms associated with
\(\tau_1,\ldots,\tau_n\) produces the elementary
differential
$$
D^nF_a(y)
\left(
F(\tau_1)(y),\ldots,F(\tau_n)(y)
\right)
=
F([\tau_1\cdots\tau_n]_a)(y)
=
F(\tau)(y).
$$
The corresponding numerical coefficient is
$$
\sum_{i=1}^{q}
(\zeta^{(a)}_{u,v})_{i}
\prod_{\ell=1}^{n}
\left(\Phi_{u,v}(\tau_\ell)\right)_i.
$$
By the definition of the componentwise product and
Definition~\ref{def:rk-tree-coefficients}, we have
$$
\sum_{i=1}^{q}
(\zeta^{(a)}_{u,v})_{i}
\prod_{\ell=1}^{n}
\left(\Phi_{u,v}(\tau_\ell)\right)_i
 =
\left\langle
\zeta_{u,v}^{(a)},
\Phi_{u,v}(\tau_1)\odot\cdots\odot
\Phi_{u,v}(\tau_n)
\right\rangle_{\mathbb R^q}
=
a_{u,v}(\tau).
$$
The same calculation of the symmetry-factor as in the stage
expansion shows that the corresponding combinatorial
factor is \(1/\sigma(\tau)\). Therefore, the complete
contribution associated with \(\tau\) is
$
\frac{a_{u,v}(\tau)}{\sigma(\tau)}F(\tau)(y).
$
It remains to estimate the omitted terms. Every term containing the stage remainder $\mathcal{R}_{i, u, v}^{(p+1)}(y)$ is multiplied by $\zeta_{u, v}^{(a)}=O(|v-u|^\alpha)$, and therefore has order at least $O(|v-u|^{(p+2) \alpha})$. 

Moreover, for any unretained tree term, we have
$
1+\sum_{\ell=1}^n\left|\tau_{\ell}\right| \geq p+1.
$
By Assumption~\ref{ass:rk-solvability-size} and
Proposition~\ref{prop:rk-tree-coefficient-scaling}, we obtain
$$
|
\sum_{i=1}^{q}
(\zeta^{(a)}_{u,v})_{i}
\prod_{\ell=1}^{n}
\left(\Phi_{u,v}(\tau_\ell)\right)_i
|
\leq
q|\zeta_{u,v}^{(a)}|_\infty
\prod_{\ell=1}^{n}
|\Phi_{u,v}(\tau_\ell)|_\infty
 \leq
C|v-u|^{
(1+\sum_{\ell=1}^{n}|\tau_\ell|)\alpha
}
\leq
C_{p,T}|v-u|^{(p+1)\alpha}.
$$
The uniform boundedness of all derivatives and elementary differentials
appearing in the finite expansion up to order \(p\), which follows from
\(F_a\in C_b^{p+1}\), combined with the Taylor-remainder estimate,
gives  
\[
\Psi_{u,v}^{\mathrm{CRK}}(y)
=
y+
\sum_{\tau\in\mathcal T_{\le p}}
\frac{a_{u,v}(\tau)}{\sigma(\tau)}
F(\tau)(y)
+
\mathcal R_{u,v}^{(p+1)}(y),
\]
where
$$
|
\mathcal R_{u,v}^{(p+1)}(y)
|
\leq
C_{p,R,T}|v-u|^{(p+1)\alpha}.
$$
Combining the estimate for $\mathcal{R}_{u, v}^{(p+1)}(y)$ with \eqref{eq:stage-bseries-expansion} gives
\[
\max_{1\leq i\leq q}
|
\mathcal R_{i,u,v}^{(p+1)}(y)
|
+
|
\mathcal R_{u,v}^{(p+1)}(y)
|
\leq
C_{p,R,T}|v-u|^{(p+1)\alpha}.
\]
This proves \eqref{eq:rk-bseries-expansion} and
\eqref{eq:rk-bseries-remainder}.
\end{proof}

We now compare the exact controlled B-series expansion of the solution to the controlled-driven rough differential equation \meqref{eq:controlled-driven-rde} with the controlled B-series expansion of its controlled Runge-Kutta one-step approximation, as given respectively in Theorem~\ref{thm:exact-controlled-bseries} and Theorem~\ref{thm:rk-bseries}.
This comparison yields the following tree matching condition.
Recall that the controlled Runge-Kutta method is given in
Definition~\ref{def:controlled-rk}.

\begin{definition}
\label{def:tree-order}
Let $p\in\mathbb{Z}_{\geq 1}$. We say that \emph{the controlled
Runge-Kutta method has tree order \(p\)} if
\begin{equation}
\label{eq:tree-order-condition}
a_{u,v}(\tau)
=
 Z^\tau_{u,v},  
\qquad
\tau\in\mathcal T_{\leq p},
\qquad
0\le u\le v\le T.
\end{equation}
\end{definition}

\begin{theorem} (Local truncation error) \label{thm:local-truncation-error-rk}
Let  
\((Y,Y')\in \mathcal D_\bX^\alpha([0,T];\mathbb R^e)\) be the exact solution of
\eqref{eq:component-controlled-rde}. Suppose that the controlled Runge-Kutta
method has tree order \(p\).
Then there exists a constant $C_{p, Y, T}>0$ such that, for every sufficiently small interval $[u, v] \subset$ $[0, T]$, the one-step approximation with initial value $Y_u$ satisfies 
\begin{equation}
\label{eq:local-truncation-error-rk}
|
Y_v-\Psi_{u,v}^{\mathrm{CRK}}(Y_u)
|
\le
C_{p,Y,T}|v-u|^{(p+1)\alpha}.
\end{equation}
\end{theorem}

\begin{proof}
Since
$
(Y,Y')\in\mathcal D_\bX^\alpha([0,T];\mathbb R^e),
$
the path \(Y\) is continuous on \([0,T]\). Define
$$
R_Y:=1+\max_{0\leq r\leq T}|Y_r|.
$$
Then \(R_Y>0\) and \(|Y_u|\leq R_Y\) for every
\(u\in[0,T]\).
By \eqref{eq:exact-controlled-bseries-expansion}, applied with
\(s=u\) and \(t=v\), we have
\begin{equation}
\label{eq:exact-expansion-local-error-proof}
Y_v
=
Y_u
+
\sum_{\tau\in\mathcal T_{\le p}}
\frac{Z_{u,v}^{\tau}}{\sigma(\tau)}
F(\tau)(Y_u)
+
R_{u,v}^{(p+1)},
\end{equation}
where
$$
|R_{u,v}^{(p+1)}|
\leq
C_{p,T}|v-u|^{(p+1)\alpha}.
$$
By Theorem~\ref{thm:rk-bseries}, applied with
\(y=Y_u\) and \(R=R_Y\), we also have
\begin{equation}
\label{eq:rk-expansion-local-error-proof}
\Psi_{u,v}^{\mathrm{CRK}}(Y_u)
=
Y_u
+
\sum_{\tau\in\mathcal T_{\le p}}
\frac{a_{u,v}(\tau)}{\sigma(\tau)}
F(\tau)(Y_u)
+
\mathcal R_{u,v}^{(p+1)}(Y_u),
\end{equation}
where
$$
|
\mathcal R_{u,v}^{(p+1)}(Y_u)
|
\leq
C_{p,R_Y,T}|v-u|^{(p+1)\alpha}.
$$
Subtracting
\eqref{eq:rk-expansion-local-error-proof}
from
\eqref{eq:exact-expansion-local-error-proof}
gives
$$
Y_v-\Psi_{u,v}^{\mathrm{CRK}}(Y_u)
=
\sum_{\tau\in\mathcal T_{\le p}}
\frac{
Z_{u,v}^{\tau}-a_{u,v}(\tau)
}{
\sigma(\tau)
}
F(\tau)(Y_u)
+
R_{u,v}^{(p+1)}
-
\mathcal R_{u,v}^{(p+1)}(Y_u).
$$
By the tree order condition in
Definition~\ref{def:tree-order},
$
a_{u,v}(\tau)= Z_{u,v}^{\tau}
$
for all $\tau\in\mathcal T_{\le p}$.
Therefore,
$$
|
Y_v-\Psi_{u,v}^{\mathrm{CRK}}(Y_u)
|
\leq
|R_{u,v}^{(p+1)}|
+
|
\mathcal R_{u,v}^{(p+1)}(Y_u)
|
\leq
(
C_{p,T}+C_{p,R_Y,T}
)
|v-u|^{(p+1)\alpha}.
$$
Taking
$
C_{p,Y,T}:=C_{p,T}+C_{p,R_Y,T}
$, this completes the proof of  \eqref{eq:local-truncation-error-rk}.
\end{proof}
 
The conclusion of Theorem~\ref{thm:local-truncation-error-rk} remains valid if
the exact tree matching condition \eqref{eq:tree-order-condition} is replaced by
the approximate condition
\begin{equation}
|
a_{u,v}(\tau)- Z^\tau_{u,v}
|
\le
C_{p,T}|v-u|^{(p+1)\alpha},
\qquad
\tau\in\mathcal T_{\le p},
\qquad
0\le u\le v\le T.
\label{eq:appr}
\end{equation}
Indeed, after subtracting the two B-series expansions, the additional tree
contribution is then also of order \(|v-u|^{(p+1)\alpha}\), since
\(\mathcal T_{\le p}\) is finite and the elementary differentials are bounded on
bounded sets. The approximate form~\eqref{eq:appr} will be useful when considering the simplified
schemes in the next section. For the construction of Runge-Kutta schemes, we now return to the condition \eqref{eq:tree-order-condition} and derive its explicit form for low-order methods. 

\begin{proposition}
\label{prop:tree-order-conditions-order-three}
For \(p=3\), the exact tree matching condition 
\eqref{eq:tree-order-condition} is equivalent to the following identities, for every
\(0\le u\le v\le T\) and every \(a,b,c\in  \mathcal A\), 
\begin{equation}
\label{eq:order-one-condition}
\sum_{i=1}^q
\zeta_{u,v,i}^{(a)}
=
 Z_{u,v}^{\bullet_a},\qquad
\sum_{i,j=1}^q
(\zeta_{u,v}^{(a)})_{i}
(\mathcal Z_{u,v}^{(b)})_{i,j}
=
 Z_{u,v}^{[\bullet_b]_a},
\end{equation}

\begin{equation}
\label{eq:order-three-chain-condition}
\sum_{i,j,k=1}^q
(\zeta_{u,v}^{(a)})_{i}
(\mathcal Z_{u,v}^{(b)})_{i,j}
(\mathcal Z_{u,v}^{(c)})_{j,k}
=
 Z_{u,v}^{[[\bullet_c]_b]_a},
\qquad
\sum_{i=1}^q
(\zeta_{u,v}^{(a)})_{i}
\left(
\sum_{j=1}^q
(\mathcal Z_{u,v}^{(b)})_{i,j}
\right)
\left(
\sum_{k=1}^q
(\mathcal Z_{u,v}^{(c)})_{i,k}
\right)
=
Z_{u,v}^{[\bullet_b\bullet_c]_a}.
\end{equation}
\end{proposition}

\begin{proof}
By Definition~\ref{def:rk-tree-coefficients}, we have
\[
\Phi_{u,v}(\mathbf 1)=\mathbf 1_q,
\qquad
\Phi_{u,v}(\bullet_b)
=
\mathcal Z_{u,v}^{(b)}\mathbf 1_q,
\qquad
\Phi_{u,v}([\bullet_c]_b)
=
\mathcal Z_{u,v}^{(b)}
\mathcal Z_{u,v}^{(c)}
\mathbf 1_q.
\]
Consequently, the corresponding coefficients are
\[ a_{u,v}(\bullet_a) = \left\langle \zeta_{u,v}^{(a)},\mathbf 1_q \right\rangle_{\mathbb R^q},\quad a_{u,v}([\bullet_b]_a) = \left\langle \zeta_{u,v}^{(a)}, \mathcal Z_{u,v}^{(b)}\mathbf 1_q \right\rangle_{\mathbb R^q},\quad a_{u,v}([[\bullet_c]_b]_a) = \left\langle \zeta_{u,v}^{(a)}, \mathcal Z_{u,v}^{(b)} \mathcal Z_{u,v}^{(c)} \mathbf 1_q \right\rangle_{\mathbb R^q} \] and \[ a_{u,v}([\bullet_b\bullet_c]_a) = \left\langle \zeta_{u,v}^{(a)}, \bigl(\mathcal Z_{u,v}^{(b)}\mathbf 1_q\bigr) \odot \bigl(\mathcal Z_{u,v}^{(c)}\mathbf 1_q\bigr) \right\rangle_{\mathbb R^q}. \] 
Expanding the inner products and matrix products yields
\eqref{eq:order-one-condition}--\eqref{eq:order-three-chain-condition}, respectively.
Since every nonempty decorated rooted tree of order at most three is of one of the forms
\[
\bullet_a,\qquad
[\bullet_b]_a,\qquad
[[\bullet_c]_b]_a,\qquad
[\bullet_b\bullet_c]_a,
\]
where \(a,b,c\in\mathcal A\), the relations in \eqref{eq:order-one-condition}--\eqref{eq:order-three-chain-condition} are
equivalent to
$a_{u, v}(\tau)=Z_{u, v}^\tau$ for all  $\tau \in \mathcal{T}_{\leq 3}$ with $|\tau| \geq 1 $.
\end{proof}
\begin{remark}
\label{rem:relation-classical-branched}
Let \(\mathbf X^{\mathrm{br}}\) be a lifted branched rough path  of \(X\) and
\(
\mathbf Z = (Z,Z')=(X,\operatorname{Id}_{\mathbb R^d}).
\)
Then the coefficients \(Z^\tau_{u,v}\) given in Definition~\ref{def:controlled-forest-coefficients} coincide with the
corresponding branched rough path coordinates:
\(
 Z^\tau_{u,v}
=
\langle \mathbf X^{\mathrm{br}}_{u,v},\tau\rangle\) for
\(\tau\in\mathcal T_{\leq p}.
\)
Thus, the tree matching
condition
\(
a_{u,v}(\tau)=Z^\tau_{u,v}\) for
\(\tau\in\mathcal T_{\leq p},
\)
reduces to the classical Runge-Kutta order condition for rough differential
equations driven by branched rough paths.
\end{remark}

\section{Simplified controlled Runge-Kutta methods}
\label{sec:simplified-rk}
In this section, we introduce a simplified controlled Runge-Kutta method based on a piecewise linear approximation of the effective driver $Z$. The controlled Runge-Kutta method developed in Section~\ref{sec:runge-kutta-methods} involves coefficients $Z^\tau_{u,v}$ defined recursively through controlled rough integration. The piecewise linear approximation replaces these coefficients by explicit increment-based expressions, leading to a more directly implementable scheme in the spirit of simplified methods for rough differential equations \cite{DNT2012,FV2010}. 
 
Let
$
\pi_h:=\{0=t_0<t_1<\cdots<t_N=T\}
$
be a partition of \([0,T]\). We write
$
h_n:=t_{n+1}-t_n
$ and $
h:=\max_{0\le n\le N-1}h_n .
$
Define the piecewise (smooth) linear interpolation $Z^h$ of \(Z\) by 
\begin{equation}\label{eq:piecewise-linear-Z}
Z_t^h:=Z_{t_n}+\frac{t-t_n}{h_n}Z_{t_n,t_{n+1}},\qquad t\in (t_n,t_{n+1}], \qquad 0\leq n\leq N-1,
\end{equation}
with $Z_0^h:=Z_0.$ Such piecewise linear and discrete approximations are standard tools in the
numerical analysis of rough differential equations; see, for example,
\cite{Davie2008}. Since \(Z^h\) is linear on each interval
\([t_n,t_{n+1}]\), it follows that
\[
dZ_t^h=\frac{Z_{t_n,t_{n+1}}}{h_n}\,dt,
\qquad t\in(t_n,t_{n+1}).  
\]
Let
\(Y^h\) be the solution driven by \(Z^h\), satisfying
\begin{equation}
\label{eq:piecewise-linear-rde}
dY_t^h=F(Y_t^h)\,dZ_t^h,
\qquad
Y_0^h=y_0.
\end{equation}
On each interval \((t_n,t_{n+1})\), the path \(Z^h\) is differentiable, and
\eqref{eq:piecewise-linear-rde} is equivalent to the ordinary differential equation
\begin{equation}
\label{E5-3}
\frac{dY_t^h}{dt}
=
F(Y_t^h)
\frac{Z_{t_n,t_{n+1}}}{h_n}
=
\sum_{a=1}^m
F_a(Y_t^h)
\frac{Z^a_{t_n,t_{n+1}}}{h_n},
\qquad
t\in(t_n,t_{n+1}),
\end{equation}
where \(Z^a_{t_n,t_{n+1}}\) denotes the \(a\)-th component of the increment
\(Z_{t_n,t_{n+1}}\). 
We use the notation
\[
\Delta Z_n:=Z_{t_n,t_{n+1}},
\qquad
\Delta Z_{u,v}:=Z_v-Z_u,
\qquad
\Delta Z_n^a:=Z^a_{t_n,t_{n+1}},\qquad 
a\in\mathcal A.
\]

For each fixed partition $\pi_h = \{0=t_0<t_1<\cdots<t_N=T\}$, regard the piecewise linear
interpolation $Z^h$ as the $\mathbf X$-controlled path
$\mathbf Z^h:=(Z^h,0)$, equipped with the zero Gubinelli derivative.
On each mesh interval $[t_n,t_{n+1}]$, define the coefficients
$Z_{t_n,t_{n+1}}^{h,\tau}$ by the recursion in
Definition~\ref{def:controlled-forest-coefficients}, applied to the
restriction of $\mathbf Z^h$ to that interval. In this case, the
controlled rough integrals coincide with the ordinary
Riemann-Stieltjes integrals against $Z^h$.

\begin{proposition}
\label{prop:tree-coefficients-piecewise-linear} 
For every $\tau\in\mathcal T$ and every mesh interval
$[t_n,t_{n+1}]$ of $\pi_h$, we have
\begin{equation}
\label{eq:piecewise-linear-tree-coefficient}
Z_{t_n,t_{n+1}}^{h,\tau}
=
\frac{1}{\gamma(\tau)}
\prod_{v\in V(\tau)}
\Delta Z_n^{\ell(v)},
\end{equation}
where \(V(\tau)\) is the set of vertices of 
$\tau$, \(\ell(v)\in \mathcal{A}\) denotes the label of \(v\) and \(\gamma(\tau)\) is the tree factorial from Definition~\ref{def:symmetry-factor}.
\end{proposition}

\begin{proof}
Fix a grid interval
\([u,v]=[t_n,t_{n+1}]\) and write \(h_n=v-u = t_{n+1} -t_n\).
We prove by induction on \(|\tau|\geq 1\) the stronger identity
\[
 Z_{u,r}^{h,\tau}
=
\frac{1}{\gamma(\tau)}
\left(\frac{r-u}{h_n}\right)^{|\tau|}
\prod_{w\in V(\tau)}
\Delta Z_n^{\ell(w)},
\qquad r\in[u,v].
\]
Taking \(r=v\)  gives
\eqref{eq:piecewise-linear-tree-coefficient}.
For the one-node tree \(\tau=\bullet_a\), by Definition~\ref{def:controlled-forest-coefficients} and \eqref{eq:piecewise-linear-Z}, we have 
$$
 Z_{u,r}^{h,\bullet_a}
=
\sint_u^r d\bZ_s^{h,a}
=
Z_{u,r}^{h,a}
=
\frac{r-u}{h_n}\Delta Z_n^a.
$$
Since \(\gamma(\bullet_a)=1\), the identity  holds for \(|\tau|=1\).

Suppose that it holds for all trees of order strictly smaller
than \(|\tau|\), and write 
$
\tau=[\eta]_a$ with
$
\eta =\tau_1\cdots\tau_k 
$ and $\tau_1, \ldots,\tau_k\in \mathcal{T}$.
By the induction hypothesis, for \(s\in[u,v]\),
\[
 Z_{u,s}^{h,\tau_j}
=
\frac{1}{\gamma(\tau_j)}
\left(\frac{s-u}{h_n}\right)^{|\tau_j|}
\prod_{w\in V(\tau_j)}
\Delta Z_n^{\ell(w)}.
\]
Therefore,
\[
Z_{u,s}^{h,\eta}
=
\left(
\prod_{j=1}^k\frac{1}{\gamma(\tau_j)}
\right)
\left(\frac{s-u}{h_n}\right)^{|\eta|}
\prod_{w\in V(\eta)}
\Delta Z_n^{\ell(w)}.
\]
Since
$
dZ_s^{h,a}
=
\frac{\Delta Z_n^a}{h_n}\,ds,
$
Definition~\ref{def:controlled-forest-coefficients} gives
\[
\begin{aligned}
 Z_{u,r}^{h,\tau}
&=
\sint_u^r
 Z_{u,s}^{h,\eta}\,d\bZ_s^{h,a}
\\
&=
\left(
\prod_{j=1}^k\frac{1}{\gamma(\tau_j)}
\right)
\prod_{w\in V(\eta)}
\Delta Z_n^{\ell(w)}
\Delta Z_n^a
\sint_u^r
\left(\frac{s-u}{h_n}\right)^{|\eta|}
\frac{ds}{h_n}
\\
&=
\frac{1}{|\eta|+1}
\left(
\prod_{j=1}^k\frac{1}{\gamma(\tau_j)}
\right)
\left(\frac{r-u}{h_n}\right)^{|\eta|+1}
\prod_{w\in V(\tau)}
\Delta Z_n^{\ell(w)}.
\end{aligned}
\]
Since
$
|\tau|=|\eta|+1
$ and $
\gamma(\tau)
=
|\tau|
\prod_{j=1}^k\gamma(\tau_j),
$
it follows that
\[
 Z_{u,r}^{h,\tau}
=
\frac{1}{\gamma(\tau)}
\left(\frac{r-u}{h_n}\right)^{|\tau|}
\prod_{w\in V(\tau)}
\Delta Z_n^{\ell(w)}.
\]
This completes the induction. Taking \(r=v\) and using
\((v-u)/h_n=1\) yields
\[
 Z_{t_n,t_{n+1}}^{h,\tau}
=
\frac{1}{\gamma(\tau)}
\prod_{w\in V(\tau)}
\Delta Z_n^{\ell(w)},
\]
which proves \eqref{eq:piecewise-linear-tree-coefficient}.
\end{proof}
After replacing \(Z\) by its piecewise linear interpolation, the coefficients on each grid interval are explicit functions of the first-level increments \(\Delta Z_n\).

\begin{corollary}
\label{cor:exact-bseries-piecewise-linear}
Let \(Y^h\) be the solution of \eqref{eq:piecewise-linear-rde}. Assume that $(p+1)\alpha>1$ and $F_a\in C_b^{p+1}(\mathbb R^e;\mathbb R^e)$
for $a\in\mathcal A$. Then, on each
grid interval \([t_n,t_{n+1}]\), we have 
\begin{equation}
\label{eq:exact-piecewise-linear-bseries}
Y_{t_{n+1}}^h
=
Y_{t_n}^h
+
\sum_{\tau\in\mathcal T_{\le p}}
\frac{1}{\sigma(\tau)\gamma(\tau)}
F(\tau)(Y_{t_n}^h)
\prod_{v\in V(\tau)}
\Delta Z_n^{\ell(v)}
+
R_n^{(p+1),h},
\end{equation}
where
$
|R_n^{(p+1),h}|
\le
C_{p,T}h_n^{(p+1)\alpha}
$ and  \(F(\tau)\) is the elementary differential associated with the
vector fields \(F_a\) defined in Definition~\ref{def:elementary-differentials}.
\end{corollary}

\begin{proof}
Fix a grid interval \([t_n,t_{n+1}]\) and write
$
h_n:=t_{n+1}-t_n.
$
Since \(Z^h\) is linear on \([t_n,t_{n+1}]\), it follows that
$
dZ_t^{h,a}
=
\frac{\Delta Z_n^a}{h_n}\,dt.
$
Hence, \eqref{eq:piecewise-linear-rde} reduces on
\([t_n,t_{n+1}]\) to the  ordinary differential equation
\begin{equation}\label{eq:local-piecewise-linear-ode}
\frac{d}{dt}Y_t^h
=
\sum_{a=1}^m
F_a(Y_t^h)\frac{\Delta Z_n^a}{h_n},
\qquad
Y^h(t_n)=Y_{t_n}^h.
\end{equation}
The usual rooted-tree Taylor expansion gives
\[
Y_{t_{n+1}}^h
=
Y_{t_n}^h
+
\sum_{\tau\in\mathcal T_{\le p}}
\frac{1}{\sigma(\tau)}
F(\tau)(Y_{t_n}^h)
 Z_{t_n,t_{n+1}}^{h,\tau}
+
R_n^{(p+1),h},
\]
where \(Z_{t_n,t_{n+1}}^{h,\tau}\) denotes the
iterated-integral coefficient generated by the piecewise linear path \(Z^h\) on
\([t_n,t_{n+1}]\). The identification of elementary differentials and symmetry 
factors is the same as in the proof of
Theorem~\ref{thm:exact-controlled-bseries}.
To estimate the remainder, set
$
\theta=\frac{t-t_n}{h_n}.
$
Then \eqref{eq:local-piecewise-linear-ode} on \([0,1]\) becomes
\[
\frac{d}{d\theta}Y^h
=
\sum_{a=1}^m
F_a(Y^h)\Delta Z_n^a.
\]
Therefore, the remainder after truncation in tree order \(p\) satisfies
$
|R_n^{(p+1),h}|
\leq
C_{p,T}|\Delta Z_n|^{p+1}.
$
Since
$
|\Delta Z_n|
\leq
\|Z\|_{\alpha;[0,T]}h_n^\alpha,
$
it follows that
\[
|R_n^{(p+1),h}|
\leq
C_{p,T}h_n^{(p+1)\alpha}.
\]
By Proposition~\ref{prop:tree-coefficients-piecewise-linear},
for every \(\tau\in\mathcal T_{\le p}\), we have
\[
 Z_{t_n,t_{n+1}}^{h,\tau}
=
\frac{1}{\gamma(\tau)}
\prod_{v\in V(\tau)}
\Delta Z_n^{\ell(v)}.
\]
Substitution into the preceding B-series expansion yields
\[
Y_{t_{n+1}}^h
=
Y_{t_n}^h
+
\sum_{\tau\in\mathcal T_{\le p}}
\frac{1}{\sigma(\tau)\gamma(\tau)}
F(\tau)(Y_{t_n}^h)
\prod_{v\in V(\tau)}
\Delta Z_n^{\ell(v)}
+
R_n^{(p+1),h}.
\]
This proves
\eqref{eq:exact-piecewise-linear-bseries}.
\end{proof}

We now consider a simplified controlled Runge-Kutta scheme, 
which is a special case of the general controlled
Runge-Kutta method in Definition~\ref{def:controlled-rk} by setting
\begin{equation}
\label{eq:simplified-specialization}
\mathcal Z_{u,v}^{(a)}
=
\Delta Z_{u,v}^a A,
\qquad
\zeta_{u,v}^{(a)}
=
\Delta Z_{u,v}^a b,
\qquad
a\in\mathcal A.
\end{equation}

\begin{definition} (Simplified controlled Runge-Kutta method)
\label{def:simplified-controlled-rk}
Fix \(q\in\mathbb N\), a matrix
$
A=(a_{ij})_{1\le i,j\le q}\in\mathbb R^{q\times q}
$
and a weight vector
$
b=(b_1,\ldots,b_q)^\top\in\mathbb R^q.
$
For a step \([u,v]\subset[0,T]\) and \(y\in\mathbb R^e\), the \emph{stage values}
\(Y_i^{u,v}(y)\) are defined by
\begin{equation}
\label{eq:simplified-rk-stages}
Y_i^{u,v}(y)
:=
y+
\sum_{j=1}^q
a_{ij}
F\bigl(Y_j^{u,v}(y)\bigr)\Delta Z_{u,v},
\qquad
i=1,\ldots,q,
\end{equation}
where \(F(y)\Delta Z_{u,v}\) is the action of the linear map \(F(y)\in\mathcal L(\mathbb R^m,\mathbb R^e)\)  on the vector \(\Delta Z_{u,v}\in\mathbb R^m\).
The corresponding {\it simplified controlled Runge-Kutta map} is  
\begin{equation}
\label{eq:simplified-rk-one-step}
\Psi_{u,v}^{\mathrm{SCRK}}(y)
:=
y+
\sum_{i=1}^q
b_i
F\bigl(Y_i^{u,v}(y)\bigr)\Delta Z_{u,v}.
\end{equation}
\end{definition}

Equivalently, using \(F_a(y):=F(y)e_a\), \eqref{eq:simplified-rk-stages} and
\eqref{eq:simplified-rk-one-step} can be written componentwise as
\[
Y_i^{u,v}(y)
=
y+
\sum_{a=1}^m
\sum_{j=1}^q
a_{ij}\Delta Z_{u,v}^a
F_a\bigl(Y_j^{u,v}(y)\bigr),
\qquad
\Psi_{u,v}^{\mathrm{SCRK}}(y)
=
y+
\sum_{a=1}^m
\sum_{i=1}^q
b_i\Delta Z_{u,v}^a
F_a\bigl(Y_i^{u,v}(y)\bigr).
\]
We also define
$
c_i:=\sum_{j=1}^q a_{ij},
\,
i=1,\ldots,q.
$

\begin{remark}
\label{rem:explicit-implicit-srk}
If $A$ is strictly lower triangular, then the stage equations \eqref{eq:simplified-rk-stages} are explicit. Otherwise, the scheme is implicit, and we assume that the stage system is solvable at each step.
\end{remark}


The simplified scheme in Definition~\ref{def:simplified-controlled-rk} is obtained
from the general controlled Runge-Kutta framework by using coefficients depending
only on the first-level increment \(\Delta Z_{u,v}\). We now identify the resulting
coefficients.
By Definition~\ref{def:rk-tree-coefficients}, we define recursively
\[
\widehat\Phi:\mathcal T^0\to\mathbb R^q,
\qquad
\widehat a:\mathcal T^0\to\mathbb R
\]
by 
\begin{equation}
\label{eq:simplified-stage-weight}
\widehat\Phi(\etree):=\mathbf 1_q,
\qquad
\widehat\Phi([\tau_1\cdots\tau_n]_a)
:=
A\bigl(\widehat\Phi(\tau_1)\odot\cdots\odot\widehat\Phi(\tau_n)\bigr)
\end{equation}
and
\begin{equation}
\label{eq:simplified-output-weight}
\widehat a(\etree):=1,
\qquad
\widehat a([\tau_1\cdots\tau_n]_a)
:=
\bigl\langle
b,\,
\widehat\Phi(\tau_1)\odot\cdots\odot\widehat\Phi(\tau_n)
\bigr\rangle_{\mathbb R^q}.
\end{equation}


\begin{proposition}
\label{prop:simplified-factorization}
For \(\tau\in\mathcal T\) and \(0\le u\le v\le T\), let
$\Phi_{u,v}^{\mathrm{SCRK}}:\mathcal T^0\to\mathbb R^q$ and $a_{u,v}^{\mathrm{SCRK}}:\mathcal T^0\to\mathbb R$ denote the
stage and output coefficients associated with the simplified scheme
\eqref{eq:simplified-rk-stages} and~\eqref{eq:simplified-rk-one-step}, respectively. Then, for $i=1,\ldots, q,$
\begin{equation}
\label{eq:simplified-stage-factorization}
\bigl(\Phi_{u,v}^{\mathrm{SCRK}}(\tau)\bigr)_i
=
\bigl(\widehat\Phi(\tau)\bigr)_i
\prod_{w\in V(\tau)}
\Delta Z_{u,v}^{\ell(w)},
\qquad
a_{u,v}^{\mathrm{SCRK}}(\tau)
=
\widehat a(\tau)
\prod_{w\in V(\tau)}
\Delta Z_{u,v}^{\ell(w)}.
\end{equation}
\end{proposition}

\begin{proof}

We prove \eqref{eq:simplified-stage-factorization} by induction for \(|\tau|\geq 1\).
For the one-node tree $\tau=\bullet_a$ with $a\in \mathcal{A}$, applying Definition~
\ref{def:rk-tree-coefficients} together with \eqref {eq:simplified-stage-weight}, \eqref {eq:simplified-output-weight} and
\eqref{eq:simplified-specialization} yields
\begin{align*}
\Phi_{u,v}^{\mathrm{SCRK}}(\bullet_a)
=&\ 
\mathcal Z_{u,v}^{(a)}\mathbf 1_q
=
\Delta Z_{u,v}^a A\mathbf 1_q
=
\widehat\Phi(\bullet_a)\Delta Z_{u,v}^a, \\
a_{u,v}^{\mathrm{SCRK}}(\bullet_a)
=&\ 
\left\langle
\zeta_{u,v}^{(a)},\mathbf 1_q
\right\rangle_{\mathbb R^q}
=
\Delta Z_{u,v}^a
\left\langle
b,\mathbf 1_q
\right\rangle_{\mathbb R^q}
=
\widehat a(\bullet_a)\Delta Z_{u,v}^a.
\end{align*}

Since \(V(\bullet_a)\) consists only of the root labeled by \(a\),
both identities hold for \(|\tau|=1\).
Suppose that the identities hold for all trees of order strictly
smaller than \(|\tau|\), and let
$
\tau=[\tau_1\cdots\tau_n]_a$ with $|\tau|\geq2.
$
By Definition~\ref{def:rk-tree-coefficients},
\eqref{eq:simplified-specialization} and the induction hypothesis,
we obtain
\[
\begin{aligned}
\Phi_{u,v}^{\mathrm{SCRK}}(\tau)
&=
\mathcal Z_{u,v}^{(a)}
\bigl(
\Phi_{u,v}^{\mathrm{SCRK}}(\tau_1)
\odot\cdots\odot
\Phi_{u,v}^{\mathrm{SCRK}}(\tau_n)
\bigr)
\\
&=
\Delta Z_{u,v}^a
A\bigl(
\widehat\Phi(\tau_1)
\odot\cdots\odot
\widehat\Phi(\tau_n)
\bigr)
\prod_{j=1}^n
\prod_{w\in V(\tau_j)}
\Delta Z_{u,v}^{\ell(w)}
\\
&=
\widehat\Phi(\tau)
\prod_{w\in V(\tau)}
\Delta Z_{u,v}^{\ell(w)}.
\end{aligned}
\]
Here the last equality follows from the recursive definition of
\(\widehat\Phi\) and the fact that the vertices of \(\tau\) consist
of its root, labeled by \(a\), together with the vertices of
\(\tau_1,\ldots,\tau_n\).

Similarly, by the definition of the output coefficient and the induction
hypothesis, we obtain
\[
a_{u,v}^{\mathrm{SCRK}}(\tau)
=
\widehat a(\tau)
\prod_{w\in V(\tau)}
\Delta Z_{u,v}^{\ell(w)} .
\]
This proves \eqref{eq:simplified-stage-factorization} and completes the
induction.
\end{proof}

The B-series of the simplified scheme is therefore obtained by specializing the
general controlled Runge-Kutta B-series in Theorem~\ref{thm:rk-bseries} via
\eqref{eq:simplified-specialization}.

\begin{theorem} (Controlled B-series expansion of the simplified Runge-Kutta map)
\label{thm:simplified-bseries}
For $(p+1) \alpha>1$, let $
F_a\in C_b^{p+1}(\mathbb R^e;\mathbb R^e)$ for $a\in\mathcal A.
$
Then, for every \(R>0\), there exists a constant \(C_{p,R,T}>0\) such that, for all
\(0\le u\le v\le T\) and all \(|y|\le R\),
\begin{equation}
\label{eq:simplified-bseries}
\Psi_{u,v}^{\mathrm{SCRK}}(y)
=
y
+
\sum_{\tau\in\mathcal T_{\le p}}
\frac{\widehat a(\tau)}{\sigma(\tau)}
F(\tau)(y)
\prod_{w\in V(\tau)}
\Delta Z_{u,v}^{\ell(w)}
+
\mathcal R_{u,v}^{(p+1),\mathrm{SCRK}}(y),
\end{equation}
where
$$
|
\mathcal R_{u,v}^{(p+1),\mathrm{SCRK}}(y)
|
\le
C_{p,R,T}|v-u|^{(p+1)\alpha}.
$$
\end{theorem}

\begin{proof}
The \(\alpha\)-H\"older continuity of \(Z\) yields
\[
|\Delta Z_{u,v}^a|
\leq
\|Z^a\|_{\alpha;[0,T]}|v-u|^\alpha,
\qquad a\in\mathcal A.
\]
Consequently,
$$
\|\mathcal Z_{u,v}^{(a)}\|_\infty
+
|\zeta_{u,v}^{(a)}|_\infty
=
|\Delta Z_{u,v}^a|
\bigl(
\|A\|_\infty+|b|_\infty
\bigr)
\leq
\bigl(
\|A\|_\infty+|b|_\infty
\bigr)
\|Z^a\|_{\alpha;[0,T]}
|v-u|^\alpha
\leq
C_{\mathrm{CRK}}|v-u|^\alpha.
$$
Together with the solvability condition stated in Remark~\ref{rem:explicit-implicit-srk}, Assumption~\ref{ass:rk-solvability-size} is satisfied. Applying Theorem~\ref{thm:rk-bseries} to the simplified controlled Runge-Kutta method gives
\[
\Psi_{u,v}^{\mathrm{SCRK}}(y)
=
y+
\sum_{\tau\in\mathcal T_{\le p}}
\frac{a_{u,v}^{\mathrm{SCRK}}(\tau)}{\sigma(\tau)}
F(\tau)(y)
+
\mathcal R_{u,v}^{(p+1),\mathrm{SCRK}}(y),
\]
where
$$
|
\mathcal R_{u,v}^{(p+1),\mathrm{SCRK}}(y)
|
\leq
C_{p,R,T}|v-u|^{(p+1)\alpha}.
$$
By Proposition~\ref{prop:simplified-factorization},  
$$
a_{u,v}^{\mathrm{SCRK}}(\tau)
=
\widehat a(\tau)
\prod_{w\in V(\tau)}
\Delta Z_{u,v}^{\ell(w)}.
$$
Substituting the expression for $a_{u, v}^{\mathrm{SCRK}}(\tau)$ into the preceding expansion yields
\[
\Psi_{u,v}^{\mathrm{SCRK}}(y)
=
y+
\sum_{\tau\in\mathcal T_{\le p}}
\frac{\widehat a(\tau)}{\sigma(\tau)}
F(\tau)(y)
\prod_{w\in V(\tau)}
\Delta Z_{u,v}^{\ell(w)}
+
\mathcal R_{u,v}^{(p+1),\mathrm{SCRK}}(y).
\]
Hence \eqref{eq:simplified-bseries} holds, and the remainder estimate
follows directly from Theorem~\ref{thm:rk-bseries}. This completes the
proof of Theorem~\ref{thm:simplified-bseries}.
\end{proof}



The exact and numerical controlled B-series expansions in
Corollary~\ref{cor:exact-bseries-piecewise-linear} and
Theorem~\ref{thm:simplified-bseries} share the same elementary
differential structure, while their coefficients are given by
\(\gamma(\tau)^{-1}\) and \(\widehat a(\tau)\), respectively.

\begin{definition}
\label{def:simplified-tree-order}
We say that the \emph{simplified controlled Runge-Kutta method
has tree order \(p \in \mathbb N\)} if
\begin{equation}
\label{eq:simplified-tree-order-condition}
\widehat a(\tau)
=
\frac{1}{\gamma(\tau)},
\qquad
\tau\in\mathcal T_{\le p} .
\end{equation}
\end{definition}
 
\begin{theorem}
\label{thm:simplified-local-error-piecewise-linear}
Assume that the simplified controlled
Runge-Kutta method has tree order \(p\) in the sense of
Definition~\ref{def:simplified-tree-order}. Then, for every \(R>0\), there exists
a constant \(C_{p,R,T}>0\) such that, whenever
$
|Y_{t_n}^h|\le R,
$
one has
\begin{equation}
\label{eq:simplified-local-error-piecewise-linear}
\left|
Y_{t_{n+1}}^h
-
\Psi_{t_n,t_{n+1}}^{\mathrm{SCRK}}(Y_{t_n}^h)
\right|
\le
C_{p,R,T} h_n^{(p+1)\alpha},
\end{equation}
for every grid interval \([t_n,t_{n+1}]\).
\end{theorem}

\begin{proof}
Corollary~\ref{cor:exact-bseries-piecewise-linear} gives the expansion
\[
Y_{t_{n+1}}^h
=
Y_{t_n}^h
+
\sum_{\tau\in\mathcal T_{\le p}}
\frac{1}{\sigma(\tau)\gamma(\tau)}
F(\tau)(Y_{t_n}^h)
\prod_{v\in V(\tau)}
\Delta Z_n^{\ell(v)}
+
R_n^{(p+1),h},
\]
where
$$
|R_n^{(p+1),h}|
\le
C_{p,T}h_n^{(p+1)\alpha}.
$$
On the other hand, from Theorem~\ref{thm:simplified-bseries} with
\(u=t_n\), \(v=t_{n+1}\) and \(y=Y_{t_n}^h\), it follows that
\[
\Psi_{t_n,t_{n+1}}^{\mathrm{SCRK}}(Y_{t_n}^h)
=
Y_{t_n}^h
+
\sum_{\tau\in\mathcal T_{\le p}}
\frac{\widehat a(\tau)}{\sigma(\tau)}
F(\tau)(Y_{t_n}^h)
\prod_{v\in V(\tau)}
\Delta Z_n^{\ell(v)}
+
\mathcal R_{t_n,t_{n+1}}^{(p+1),\mathrm{SCRK}}(Y_{t_n}^h),
\]
with
$$
\left|
\mathcal R_{t_n,t_{n+1}}^{(p+1),\mathrm{SCRK}}(Y_{t_n}^h)
\right|
\le
C_{p,R,T}h_n^{(p+1)\alpha}.
$$
Subtracting the two expansions yields
\[
\begin{aligned}
Y_{t_{n+1}}^h
-
\Psi_{t_n,t_{n+1}}^{\mathrm{SCRK}}(Y_{t_n}^h)
={}&
\sum_{\tau\in\mathcal T_{\le p}}
\frac{1/\gamma(\tau)-\widehat a(\tau)}{\sigma(\tau)}
F(\tau)(Y_{t_n}^h)
\prod_{v\in V(\tau)}
\Delta Z_n^{\ell(v)}
+
R_n^{(p+1),h}
-
\mathcal R_{t_n,t_{n+1}}^{(p+1),\mathrm{SCRK}}(Y_{t_n}^h).
\end{aligned}
\]
By the tree order condition \eqref{eq:simplified-tree-order-condition}, the sum
over \(\mathcal T_{\le p}\) vanishes term by term. Therefore,
\[
|
Y_{t_{n+1}}^h
-
\Psi_{t_n,t_{n+1}}^{\mathrm{SCRK}}(Y_{t_n}^h)
|
\le
|R_n^{(p+1),h}|
+
|
\mathcal R_{t_n,t_{n+1}}^{(p+1),\mathrm{SCRK}}(Y_{t_n}^h)
|
\le
C_{p,R,T}h_n^{(p+1)\alpha}.
\]
This finishes the proof of \eqref{eq:simplified-local-error-piecewise-linear}.
\end{proof}
For $\alpha \in (\frac{1}{3}, \frac{1}{2}]$, the choice $p=2$ already gives
$(p+1) \alpha-1=3 \alpha-1>0$. In the sequel, we focus on $p=3$, which yields the convergence rate $4 \alpha-1$ and connects the tree order conditions with the classical third-order Runge-Kutta conditions. The next result makes this connection explicit.
\begin{proposition}
\label{prop:simplified-third-order-conditions}
The simplified controlled Runge-Kutta method has tree order \(3\) if and only if
\begin{align}\label{E5-1}
\sum_{i=1}^q b_i = 1,\quad
\sum_{i=1}^q b_i c_i = \frac{1}{2},\quad
\sum_{i=1}^q b_i c_i^2 = \frac{1}{3},\quad
\sum_{i,j=1}^q b_i a_{ij}c_j = \frac{1}{6}.
\end{align}
\end{proposition}

\begin{proof}
By \eqref{eq:simplified-stage-weight}, \eqref{eq:simplified-output-weight} and 
\(c=A\mathbf 1_q\),  the scalar weights of the rooted trees of order at most
three are
\[
\begin{array}{c|c|c}
\tau & \widehat a(\tau) & \gamma(\tau) \\[2mm]
\hline
\bullet_a
&
\langle b,\mathbf 1_q\rangle
=
\displaystyle\sum_{i=1}^q b_i
&
1
\\[3mm]
[\bullet_b]_a
&
\langle b,c\rangle
=
\displaystyle\sum_{i=1}^q b_i c_i
&
2
\\[3mm]
[\bullet_b\bullet_c]_a
&
\langle b,c\odot c\rangle
=
\displaystyle\sum_{i=1}^q b_i c_i^2
&
3
\\[3mm]
[[\bullet_c]_b]_a
&
\langle b,Ac\rangle
=
\displaystyle\sum_{i,j=1}^q b_i a_{ij}c_j
&
6 
\end{array}
\]
These are the only rooted tree shapes of order at most three, up to vertex labels. Moreover, since \(\widehat a(\tau)\) depends only on the underlying tree shape, the simplified tree order condition
\[\widehat a(\tau)=\frac{1}{\gamma(\tau)},
\qquad
\tau\in\mathcal T_{\le 3},
 \]
is equivalent to \eqref{E5-1}.
\end{proof}

\begin{example}
\label{ex:heun-controlled}
The explicit Butcher table
\begin{equation*}
\begin{array}{c|ccc}
0   & 0   & 0   & 0 \\
\frac13 & \frac13 & 0 & 0 \\
\frac23 & 0 & \frac23 & 0 \\
\hline
& \frac14 & 0 & \frac34
\end{array}
\end{equation*}
satisfies \eqref{E5-1}.
The coefficients are chosen from a classical third-order
Runge-Kutta method; see, for example, \cite{Butcher1963}.
The resulting simplified controlled Runge-Kutta method has tree order 3.
\end{example}

The simplified scheme uses products of the increments
\(\Delta Z_{u,v}\) instead of controlled coefficients. The next
section combines the local error estimate with the approximation error
between \(Z\) and \(Z^h\) to derive the global error bound.

\section{Error Analysis and Numerical Experiments}
\label{sec:global-error}

In this section, we establish the global error bounds for the
controlled Runge-Kutta methods and validate the theoretical convergence
rates through numerical experiments.
\subsection{Global error analysis}
The local estimates in Sections~\ref{sec:runge-kutta-methods}
and~\ref{sec:simplified-rk} provide one-step error bounds for the
controlled Runge-Kutta method and the simplified method, respectively.
Set
$
\rho:=(p+1)\alpha>1 .
$
The global error bounds follow from the local estimates and the stability of the
solution flow.
Let
\[
\pi_h=\{0=t_0<t_1<\cdots<t_N=T\},
\qquad
h=\max_n\left|t_{n+1}-t_n\right|.
\]
For \(R>0\), denote \(B_R=\{y\in\mathbb R^e:|y|\le R\}\).
All constants below are independent of \(h\).

For \(0\le s\le t\le T\), let \(\varphi_{s,t}\) and
\(\varphi^h_{s,t}\) denote the solution flows of the controlled-driven
equation and its piecewise linear approximation, respectively. They
satisfy
\[
\varphi_{s,s}=\mathrm{Id},
\qquad
\varphi_{s,t}=\varphi_{r,t}\circ\varphi_{s,r},
\quad s\le r\le t .
\]
The stability estimate required below is the following Lipschitz bound
with respect to the initial condition.

\begin{proposition}
\label{prop:flow-stability}
Let
$
F\in C_b^3\bigl(\mathbb R^e;\mathcal L(\mathbb R^m,\mathbb R^e)\bigr).
$
Then, for every \(R>0\), there exists a constant \(L_{R,T}>0\) such that
\begin{equation}
\label{eq:flow-stability-controlled}
|\varphi_{s,t}(y)-\varphi_{s,t}(\widetilde y)|
\le
L_{R,T}|y-\widetilde y|,
\qquad
y,\widetilde y\in B_R,
\qquad
0\le s\le t\le T.
\end{equation}
Moreover, if the piecewise linear drivers \(Z^h\) have uniformly bounded
controlled rough path norms, then
\begin{equation}
\label{eq:flow-stability-piecewise-linear}
|\varphi^h_{s,t}(y)-\varphi^h_{s,t}(\widetilde y)|
\le
L_{R,T}|y-\widetilde y|,
\qquad
y,\widetilde y\in B_R,
\qquad
0\le s\le t\le T,
\end{equation}
with \(L_{R,T}\) independent of \(h\).
\end{proposition}

\begin{proof}
The estimate follows from the local Lipschitz continuity of the
solution map for the controlled-driven rough differential equation
with respect to the initial condition; see, e.g., \cite{LiGao2026}.
The same argument applies to the piecewise linear drivers \(Z^h\).
Since the controlled rough path norms of \(Z^h\) are uniformly bounded,
the Lipschitz constant can be chosen independently of \(h\).
\end{proof}

The following proposition shows how the local error and the stability
estimate yield a global error bound. In the subsequent applications,
the stability estimate is obtained from Proposition~\ref{prop:flow-stability},
while the local error estimate follows from the corresponding local
expansion results.

\begin{proposition}
\label{prop:local-to-global}
Let \(\{\psi_{u,v}\}_{0\le u\le v\le T}\) be a flow on \(\mathbb R^e\), that is,
$
\psi_{u,u}=\mathrm{Id},
\,
\psi_{v,w}\circ\psi_{u,v}=\psi_{u,w},
\,
0\le u\le v\le w\le T.
$
Let \(\Psi_{u,v}:\mathbb R^e\to\mathbb R^e\) be a one-step map, and let
\(\rho>1\). Suppose that there exist constants \(L_{R,T}>0\) and
\(C_{\mathrm{loc}}>0\) such that
\begin{equation}
\label{eq:abstract-flow-stability}
|\psi_{u,v}(y)-\psi_{u,v}(\widetilde y)|
\le
L_{R,T}|y-\widetilde y|,
\qquad
y,\widetilde y\in B_R,
\qquad
0\le u\le v\le T,
\end{equation}
and
\begin{equation}
\label{eq:abstract-local-defect}
|\psi_{u,v}(y)-\Psi_{u,v}(y)|
\le
C_{\mathrm{loc}}|v-u|^\rho,
\qquad
y\in B_R,
\qquad
0\le u\le v\le T.
\end{equation}
Define the reference values and the numerical sequence by
\[
Y_{t_n}^{\psi}:=\psi_{0,t_n}(Y_0),
\qquad
y_0=Y_0,
\qquad
y_{n+1}:=\Psi_{t_n,t_{n+1}}(y_n),
\qquad
n=0,\ldots,N-1.
\]
Assume that, for every \(n\), the points
$
y_n,\,
\psi_{t_n,t_{n+1}}(y_n),
\,
\Psi_{t_n,t_{n+1}}(y_n)
$
belong to \(B_R\). Then
\begin{equation}
\label{eq:abstract-global-error}
\max_{0\le n\le N}|Y_{t_n}^{\psi}-y_n|
\le
L_{R,T}C_{\mathrm{loc}}T h^{\rho-1}.
\end{equation}
\end{proposition}

\begin{proof}
Fix \(k\in\{1,\ldots,N\}\). By the flow property of \(\psi\),
the stability estimate \eqref{eq:abstract-flow-stability}, and a
telescoping argument, we have
\[
|Y_{t_k}^{\psi}-y_k|
\le
L_{R,T}
\sum_{j=0}^{k-1}
\left|
\psi_{t_j,t_{j+1}}(y_j)
-
\Psi_{t_j,t_{j+1}}(y_j)
\right|.
\]
Using the local defect estimate \eqref{eq:abstract-local-defect}, it follows that
$$
|Y_{t_k}^{\psi}-y_k|
\le
L_{R,T}C_{\mathrm{loc}}
\sum_{j=0}^{k-1}h_j^\rho .
$$
Since \(h_j\le h\) and \(\rho>1\), we obtain
$$
\sum_{j=0}^{k-1}h_j^\rho
\le
h^{\rho-1}\sum_{j=0}^{N-1}h_j
=
Th^{\rho-1}.
$$
Therefore,
\[
|Y_{t_k}^{\psi}-y_k|
\le
L_{R,T}C_{\mathrm{loc}}T h^{\rho-1}.
\]
Taking the maximum over \(k=0,\ldots,N\) proves
\eqref{eq:abstract-global-error}.
\end{proof}

For the controlled Runge-Kutta method, the numerical iterates are defined by  
\[
y_0^{\mathrm{CRK}}:=Y_0,
\qquad
y_{n+1}^{\mathrm{CRK}}
:=
\Psi_{t_n,t_{n+1}}^{\mathrm{CRK}}
(y_n^{\mathrm{CRK}}),
\qquad
n=0,\ldots,N-1.
\]
Under the hypotheses of
Theorem~\ref{thm:local-truncation-error-rk},
Proposition~\ref{prop:local-to-global} yields the following global error bound.

\begin{theorem} (Global error bound)
\label{thm:global-controlled-rk}
Let \(R>0\) be such that \(y_n^{\mathrm{CRK}}\),
\(\varphi_{t_n,t_{n+1}}(y_n^{\mathrm{CRK}})\) and
\(\Psi_{t_n,t_{n+1}}^{\mathrm{CRK}}(y_n^{\mathrm{CRK}})\)
all belong to \(B_R\) for every \(n=0,\ldots,N-1\).
Then there exists a constant \(C_{p,R,T}>0\), independent of \(h\), such that
\begin{equation}
\label{eq:global-error-controlled-rk}
\max_{0\le n\le N}
\left|
Y_{t_n}-y_n^{\mathrm{CRK}}
\right|
\le
C_{p,R,T}h^{\rho-1}.
\end{equation}
\end{theorem}

\begin{proof}
By Theorem~\ref{thm:local-truncation-error-rk}, for any
\(y\in B_R\),
\[
\left|
\varphi_{u,v}(y)-\Psi_{u,v}^{\mathrm{CRK}}(y)
\right|
\le
C_{p,R,T}|v-u|^\rho ,
\qquad
0\le u\le v\le T .
\]
Together with the stability estimate
\eqref{eq:flow-stability-controlled}, Proposition~\ref{prop:local-to-global}
applies with 
\[
\psi_{u,v}=\varphi_{u,v},
\qquad
\Psi_{u,v}=\Psi_{u,v}^{\mathrm{CRK}},
\qquad
C_{\mathrm{loc}}=C_{p,R,T}.
\]
Since \(Y_{t_n}^{\psi}=\varphi_{0,t_n}(Y_0)=Y_{t_n}\), we obtain
\[
\max_{0\le n\le N}
|Y_{t_n}-y_n^{\mathrm{CRK}}|
\le
L_{R,T}C_{p,R,T}Th^{\rho-1}.
\]
Absorbing \(L_{R,T}T\) into the constant proves
\eqref{eq:global-error-controlled-rk}.
\end{proof}


We next consider the simplified method from
Section~\ref{sec:simplified-rk}. The numerical iterates are defined by
\[
y_0^{\mathrm{SCRK}}:=Y_0,
\qquad
y_{n+1}^{\mathrm{SCRK}}
:=
\Psi_{t_n,t_{n+1}}^{\mathrm{SCRK}}
\bigl(y_n^{\mathrm{SCRK}}\bigr),
\qquad
n=0,\ldots,N-1 .
\]
Let \(Y^h\) denote the solution of the piecewise linear problem
\eqref{eq:piecewise-linear-rde}. Under the assumptions of
Theorem~\ref{thm:simplified-local-error-piecewise-linear}, the following
result gives the global discretization error.
\begin{theorem}
\label{thm:global-simplified-smooth}
Let \(R>0\) be such that \(y_n^{\mathrm{SCRK}}\),
\(\varphi^h_{t_n,t_{n+1}}(y_n^{\mathrm{SCRK}})\) and
\(\Psi_{t_n,t_{n+1}}^{\mathrm{SCRK}}(y_n^{\mathrm{SCRK}})\)
all belong to \(B_R\) for every \(n=0,\ldots,N-1\).
Then there exists a constant \(C_{p,R,T}>0\), independent of \(h\), such that
\begin{equation}
\label{eq:global-error-simplified-smooth}
\max_{0\le n\le N}
|
Y_{t_n}^h-y_n^{\mathrm{SCRK}}
|
\le
C_{p,R,T}h^{\rho-1}.
\end{equation}
\end{theorem}

\begin{proof}
By Theorem~\ref{thm:simplified-local-error-piecewise-linear}, for
\(y\in B_R\) and every grid interval
\([t_n,t_{n+1}]\), it follows that
$$
|
\varphi^h_{t_n,t_{n+1}}(y)
-
\Psi^{\mathrm{SCRK}}_{t_n,t_{n+1}}(y)
|
\le
C_{p,R,T}h_n^\rho .
$$
The stability estimate in
\eqref{eq:flow-stability-piecewise-linear} allows us to apply
Proposition~\ref{prop:local-to-global} with
$
\psi_{u,v}=\varphi^h_{u,v}$ and 
$\Psi_{u,v}=\Psi^{\mathrm{SCRK}}_{u,v}.
$
Since
$
Y_{t_n}^{\psi}
=
\varphi^h_{0,t_n}(Y_0)
=
Y_{t_n}^h,
$
we obtain
\[
\max_{0\le n\le N}
|
Y_{t_n}^h-y_n^{\mathrm{SCRK}}
|
\le
L_{R,T}C_{p,R,T}T h^{\rho-1}.
\]
Absorbing \(L_{R,T}T\) into the constant proves
\eqref{eq:global-error-simplified-smooth}.
\end{proof}

Theorem~\ref{thm:global-simplified-smooth} estimates the discretization
error \(Y^h-y^{\mathrm{SCRK}}\). The remaining approximation error
between \(Y\) and \(Y^h\) is controlled by the following Wong-Zakai type
assumption, motivated by Gaussian rough path approximation results
\cite{FR2014}.
\begin{assumption}
\label{ass:controlled-wong-zakai}
There exist constants \(C_{\mathrm{WZ}}>0\), \(r_0>0\)  and \(h_0>0\) such that,
for every partition with mesh size \(0<h\le h_0\),
\begin{equation*}
\sup_{0\le t\le T}|Y_t-Y_t^h|
\le
C_{\mathrm{WZ}}h^{r_0}.
\end{equation*}
\end{assumption}

We arrive at the global error bound of the simplified scheme.

\begin{theorem} (Global error bound of the simplified scheme)
\label{thm:global-simplified-rk}
Assume that Assumption~\ref{ass:controlled-wong-zakai} and the
hypotheses of Theorem~\ref{thm:global-simplified-smooth} hold. There exists a constant
\(C_{p,R,T}>0\), independent of \(h\), such that
\begin{equation}
\label{eq:global-error-simplified}
\max_{0\le n\le N}
|Y_{t_n}-y_n^{\mathrm{SCRK}}|
\le
C_{p,R,T}
\left(
h^{r_0}
+
h^{(p+1)\alpha-1}
\right).
\end{equation}
Consequently, as \(h\downarrow0\),
\begin{equation}
\label{eq:global-error-simplified-min}
\max_{0\le n\le N}
|Y_{t_n}-y_n^{\mathrm{SCRK}}|
=
O\!\left(
h^{\min\{r_0,(p+1)\alpha-1\}}
\right).
\end{equation}
\end{theorem}

\begin{proof}
The triangle inequality yields
\[
\max_{0\le n\le N}
|Y_{t_n}-y_n^{\mathrm{SCRK}}|
\le
\max_{0\le n\le N}
|Y_{t_n}-Y_{t_n}^{h}|
+
\max_{0\le n\le N}
|Y_{t_n}^{h}-y_n^{\mathrm{SCRK}}|.
\]
The first term is bounded by Assumption~\ref{ass:controlled-wong-zakai},
while the second one is estimated in
Theorem~\ref{thm:global-simplified-smooth}. Therefore,
\[
\max_{0\le n\le N}
|
Y_{t_n}-y_n^{\mathrm{SCRK}}
|
\le
C_{p,R,T}
(
h^{r_0}
+
h^{(p+1)\alpha-1}
).
\]
This proves \eqref{eq:global-error-simplified}. Since
\(r_0>0\) and \((p+1)\alpha-1>0\), we have 
$$
h^{r_0}+h^{(p+1)\alpha-1}
\le
2h^{\min\{r_0,(p+1)\alpha-1\}}, \qquad 0<h\le1.
$$
Hence,
\[
\max_{0\le n\le N}
|
Y_{t_n}-y_n^{\mathrm{SCRK}}
|
=
O\left(
h^{\min\{r_0,(p+1)\alpha-1\}}
\right),
\]
which proves \eqref{eq:global-error-simplified-min}.
\end{proof}

The following corollary gives the global error bound for the third-order simplified scheme.

\begin{corollary} \label{cor:global-third-order}
Suppose that Assumption~\ref{ass:controlled-wong-zakai} holds and that the simplified controlled Runge-Kutta method satisfies the third-order conditions \eqref{E5-1} for 
$
\alpha\in (\frac{1}{3},\frac{1}{2}] 
$.
Then there exists a constant
\(C_{R,T}>0\), independent of \(h\), such that
\begin{equation*}
\max_{0\le n\le N}
|Y_{t_n}-y_n^{\mathrm{SCRK}}|
\le
C_{R,T}
h^{\min\{r_0,4\alpha-1\}}.
\end{equation*}
\end{corollary}

\begin{proof}
By Proposition~\ref{prop:simplified-third-order-conditions}, the simplified
method has tree order \(p=3\). Since
$
(p+1)\alpha-1=4\alpha-1,
$
the result follows from Theorem~\ref{thm:global-simplified-rk}.
\end{proof}

\subsection{Numerical experiments}
\label{sec:numerical-experiments}
We test the simplified controlled Runge-Kutta method
with the controlled driver
\[
\mathbf Z=(Z,Z'),\qquad
Z_t=\varphi(X_t),
\qquad
\varphi(x_1,x_2)=(x_1,\sin x_2),
\]
where \(X=(X^1,X^2)\) is a two-dimensional fractional Brownian motion
with independent components and
\(H\in\{0.40,0.45,0.50\}\).
The path \(X\) is equipped with its natural Gaussian rough path lift
\cite{FN2010}.
The Gubinelli derivative of \(Z\) is
\begin{equation*}
Z'_t
=
D\varphi(X_t)
=
\begin{pmatrix}
1 & 0\\
0 & \cos(X_t^2)
\end{pmatrix}.
\end{equation*}
We consider the scalar controlled-driven rough differential equation
\begin{equation}\label{T7-1}
dY_t
=
\cos(Y_t)\,dZ_t^1
+
\sin(Y_t)\,dZ_t^2,
\qquad
Y_0=1,
\qquad
0\le t\le T,
\end{equation}
with \(T=0.25\), where
$
F(y)(z_1,z_2)
=
\cos(y)z_1+\sin(y)z_2 .
$
For a step \([t_n,t_{n+1}]\), let
\[
\Delta Z_n:=Z_{t_n,t_{n+1}}=Z_{t_{n+1}}-Z_{t_n},
\qquad
G(y,\Delta Z_n):=F(y)\Delta Z_n
=
\cos(y)\Delta Z_n^1+\sin(y)\Delta Z_n^2 .
\]
We apply Definition~\ref{def:simplified-controlled-rk}
with \(u=t_n\), \(v=t_{n+1}\)  and \(y=y_n\).
Using the third-order method in
Example~\ref{ex:heun-controlled} with the Butcher coefficients,
we obtain
\[
y_{n,1}=y_n,\qquad
y_{n,2}=y_n+\frac13G(y_n,\Delta Z_n),\qquad
y_{n,3}=y_n+\frac23G(y_{n,2},\Delta Z_n),
\]
and
\[
y_{n+1}
=
y_n+\frac14
\bigl(G(y_n,\Delta Z_n)+3G(y_{n,3},\Delta Z_n)\bigr).
\]
We employ the following reference convergence rate for the numerical
comparison. Combining Theorem~\ref{thm:global-simplified-rk} with
Assumption~\ref{ass:controlled-wong-zakai} gives
\[
\max_{0\le n\le N}
|Y_{t_n}-y_n^{\mathrm{SCRK}}|
\le
C\bigl(h^{r_0}+h^{(p+1)\alpha-1}\bigr),
\]
so the rate is determined by
\(\min\{r_0,(p+1)\alpha-1\}\).
For \(p=3\) and \(\alpha\approx H\), we take
\(2H-\frac12\) as the reference rate for the piecewise linear
approximation, following Gaussian rough path results
\cite{FR2014}. This rate is used only as a numerical reference. Therefore, the
reference slope is chosen as
\[
\rho_{\mathrm{ref}}=2H-\frac12,
\qquad H=0.40,0.45,0.50 .
\]

The fractional Brownian paths are generated by the Davies-Harte method
\cite{DH1987}. A numerical reference solution is computed on the fine grid
with \(N_{\rm ref}=2^{18}\), giving
\[
h_{\rm ref}=T/N_{\rm ref}=2^{-20},
\qquad T=2^{-2}.
\]
The numerical solutions are computed on the coarser grids
\(h=2^{-\ell}\), \(\ell=7,\ldots,15\).
A numerical reference solution is computed on the fine grid, and the
following errors are measured with respect to this reference solution.
Let
$
\pi_h=\{0=t_0<t_1<\cdots<t_{N_h}=T\}
$. For each sample path \(j=1,\ldots,M = 100\), define
\[
E_j(h)
=
\max_{0\le k\le N_h}
|y_{t_k}^{\mathrm{ref},j}-y_k^{h,j}|,
\]
where \(y_{t_k}^{\mathrm{ref},j}\) denotes the value of the reference
solution computed on the fine grid at the coarse-grid point \(t_k\). The mean error is
\[
\overline E(h)=\frac1M\sum_{j=1}^M E_j(h).
\]

The convergence rates are estimated by fitting the log--log error curves. For each sample path \(j=1,\ldots,M\), we fit
\[
\log_{10}E_j(h)
\approx
\rho_j\log_{10}h+c_j,
\qquad
\ell=11,\ldots,15,
\]
where \(\rho_j\) is the fitted convergence rate and \(c_j\) is the
intercept. The mean and standard deviation of the pathwise rates are
\[
\overline\rho_{\rm path}
:=
\frac1M\sum_{j=1}^M\rho_j,
\qquad
s_{\rm path}
:=
\left(
\frac1{M-1}
\sum_{j=1}^M
(\rho_j-\overline\rho_{\rm path})^2
\right)^{1/2}.
\]
The mean error curve is fitted by
$
\log_{10}\overline E(h)
\approx
\rho_{\rm mean}\log_{10}h+c,
$
where \(\rho_{\rm mean}\) is the corresponding fitted rate. The fitted convergence rates are summarized below.

\begin{figure}[htbp]
\centering
\includegraphics[width=\textwidth]{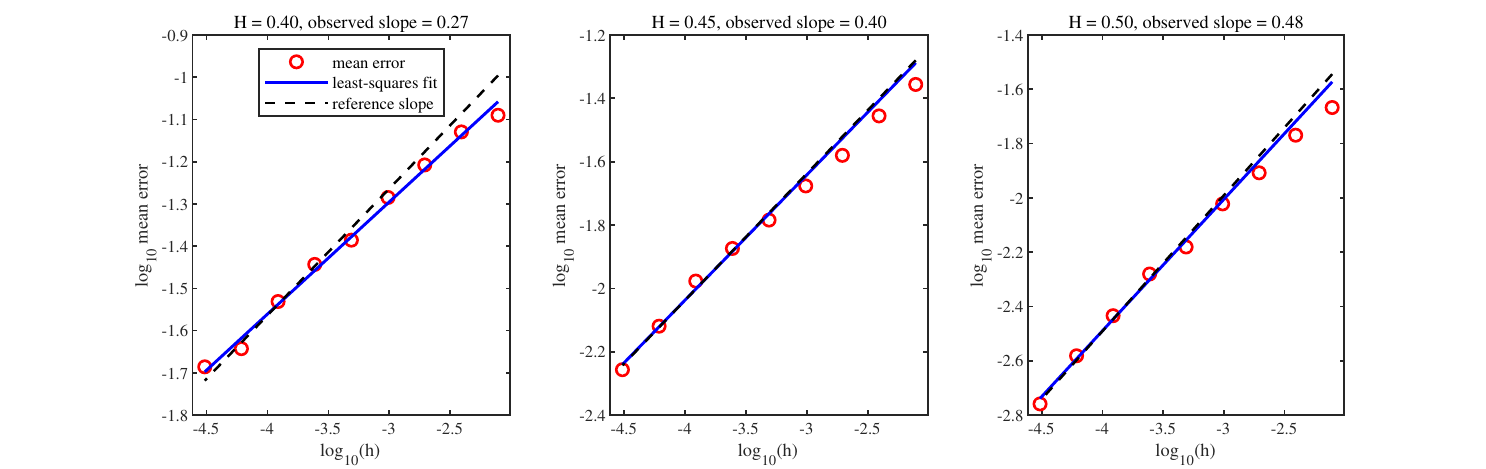}
\caption{
Mean error of the simplified controlled Runge-Kutta method for
\eqref{T7-1} with \(Z_t=(X_t^1,\sin(X_t^2))\).
Circles show \(\log_{10}\overline E(h)\) averaged over \(M=100\) sample
paths, solid lines are least-squares fits over \(\ell=11,\ldots,15\),
and dashed lines indicate the reference slope
\(\rho_{\rm ref}=2H-\frac12\).
}
\label{fig:controlled-driver-mean-error}
\end{figure}

\begin{table}[htbp]
\centering
\caption{
Observed convergence rates for the controlled-driver example
\(Z_t=(X_t^1,\sin(X_t^2))\). Here
\(\rho_{\mathrm{ref}}=2H-\frac12\) is the reference rate,
\(\overline\rho_{\mathrm{path}}\) and \(s_{\mathrm{path}}\) are the
mean and standard deviation of the pathwise fitted rates, and
\(\rho_{\mathrm{mean}}\) is the fitted rate of the mean error curve.
}
\label{tab:controlled-driver-rates}
\begin{tabular}{ccccc}
\toprule
\(H\)
& \(\rho_{\mathrm{ref}}\)
& \(\overline\rho_{\mathrm{path}}\)
& \(s_{\mathrm{path}}\)
& \(\rho_{\mathrm{mean}}\) \\
\midrule
0.40 & 0.300 & 0.267 & 0.226 & 0.266 \\
0.45 & 0.400 & 0.385 & 0.173 & 0.395 \\
0.50 & 0.500 & 0.484 & 0.234 & 0.484 \\
\bottomrule
\end{tabular}
\end{table}

The results are shown in Figure~\ref{fig:controlled-driver-mean-error}
and Table~\ref{tab:controlled-driver-rates}. The fitted rates
\(0.266,0.395,0.484\) for \(H=0.40,0.45,0.50\) agree well with the
reference rates \(2H-\frac12\). The larger deviation for \(H=0.40\)
is consistent with related rough path numerical experiments
\cite{RR2022}. Since \(2H-\frac12<4H-1\), the observed rates suggest that the
Wong-Zakai approximation error dominates in the tested regime.

\vskip 0.2in

\noindent
{\bf Acknowledgments.} Xing Gao is supported by the National Natural Science Foundation of China (12571019), the Natural Science Foundation of Gansu Province (25JRRA644) and Innovative Fundamental Research Group Project of Gansu Province (23JRRA684). Xingya Fan was supported by the National Natural Science Foundation of China (12561018).

\noindent
{\bf Declaration of interests. } The authors have no conflicts of interest to disclose.

\noindent
{\bf Data availability. } Data sharing is not applicable as no new data were created or analyzed.


\smallskip

\begin{thebibliography}{99}

\bibitem{A23}
A. Ananova, Rough differential equations with path-dependent coefficients, Ann. H. Lebesgue, {\bf 6} (2023), 1--29.

\bibitem{JDT2022}
J. Armstrong, D. Brigo, T. Cass and E. Rossi Ferrucci, Non-geometric rough paths on manifolds, J. Lond. Math. Soc., {\bf 106(2)} (2022), 756--817.

\bibitem{Bailleul2015}
I. Bailleul, Regularity of the It\^o--Lyons map, Confluentes Math.,
{\bf 7(1)} (2015), 3--11.

\bibitem{Burrage2000}
K. Burrage and P. M. Burrage, Order conditions of stochastic Runge-Kutta
methods by B-series, SIAM J. Numer. Anal., {\bf 38(5)} (2000),
1626--1646.

\bibitem{Butcher1963}
J. C. Butcher, Coefficients for the study of Runge-Kutta integration processes, J. Aust. Math. Soc.,
{\bf 3(2)} (1963), 185--201.

\bibitem{Butcher1972}
J. C. Butcher, An algebraic theory of integration methods, Math. Comp.,
{\bf 26(117)} (1972), 79--106.

\bibitem{Butcher1987}
J. C. Butcher, The Numerical Analysis of Ordinary Differential Equations,
Wiley, 1987.

\bibitem{Davie2008}
A. M. Davie, Differential equations driven by rough paths: an approach via
discrete approximation, Appl. Math. Res. Express, {\bf 2008} (2008),
abm009.

\bibitem{DH1987}
R. B. Davies and D. S. Harte, Tests for Hurst effect, Biometrika,
{\bf 74(1)} (1987), 95--101.


\bibitem{DNT2012}
A. Deya, A. Neuenkirch and S. Tindel, A Milstein-type scheme without
L\'evy area terms for SDEs driven by fractional Brownian motion,
Ann. Inst. Henri Poincar\'e{} Probab. Stat., {\bf 48(2)} (2012),
518--550.


\bibitem{FH2020}
P. K. Friz and M. Hairer, A Course on Rough Paths, Springer, 2020.

\bibitem{FR2014}
P. K. Friz and S. Riedel, Convergence rates for the full Gaussian rough
paths, Ann. Inst. Henri Poincar\'e{} Probab. Stat., {\bf 50(1)} (2014),
154--194.

\bibitem{FN2010}
P. K. Friz and N. B. Victoir, Differential equations driven by Gaussian
signals, Ann. Inst. Henri Poincar\'e{} Probab. Stat., {\bf 46(2)}
(2010), 369--413.

\bibitem{FV2010}
P. K. Friz and N. B. Victoir, Multidimensional Stochastic Processes as
Rough Paths: Theory and Applications, Cambridge University Press, 2010.

\bibitem{Gubinelli2004}
M. Gubinelli, Controlling rough paths, J. Funct. Anal., {\bf 216(1)}
(2004), 86--140.

\bibitem{Gub2010}
M. Gubinelli, Ramification of rough paths, J. Differential Equations,
{\bf 248(4)} (2010), 693--721.

\bibitem{HLW2006}
E. Hairer, C. Lubich and G. Wanner, Geometric Numerical Integration:
Structure-Preserving Algorithms for Ordinary Differential Equations,
Springer, 2006.

\bibitem{Hairer1993}
E. Hairer, S. P. N{\o}rsett and G. Wanner, Solving Ordinary Differential
Equations I: Nonstiff Problems, Springer, 1993.

\bibitem{EW1996}
E. Hairer and G. Wanner, Solving Ordinary Differential Equations II: Stiff and Differential-Algebraic Problems.
Springer, 1996.

\bibitem{HK2015}
M. Hairer and D. Kelly, Geometric versus non-geometric rough paths,
Ann. Inst. Henri Poincar\'e{} Probab. Stat., {\bf 51(1)} (2015),
207--251.


\bibitem{HHW2018}
J. Hong, C. Huang and X. Wang, Symplectic Runge-Kutta methods for
Hamiltonian systems driven by Gaussian rough paths, Appl. Numer. Math.,
{\bf 129} (2018), 120--136.


\bibitem{LiGao2026}
N. Li and X. Gao, Universal limit theorem for rough differential equations driven by controlled rough paths, preprint, 2026.

\bibitem{Lyons1998}
T. Lyons, Differential equations driven by rough signals, Rev. Mat.
Iberoamericana, {\bf 14(2)} (1998), 215--310.

\bibitem{LQ2002}
T. Lyons and Z. Qian, System Control and Rough Paths, Oxford University
Press, 2002.


\bibitem{RR2022}
M. Redmann and S. Riedel, Runge-Kutta methods for rough differential
equations, J. Stoch. Anal., {\bf 3(4)} (2022), Art. 6, 24 pp.

\end{thebibliography}
\end{document}